\documentclass[10pt]{elsarticle}
\usepackage[T1]{fontenc}
\usepackage[latin9]{inputenc}
\usepackage{url}
\usepackage{amsmath}
\usepackage{amssymb}
\usepackage{amsthm}
\usepackage{amsfonts}
\usepackage{amsopn}
\usepackage{cite}
\usepackage{mathrsfs}
\usepackage{graphicx}
\usepackage{color,colortbl}
\usepackage{algorithm}
\usepackage{algpseudocode}

\makeatletter
\@ifundefined{showcaptionsetup}{}{%
 \PassOptionsToPackage{caption=false}{subfig}}
\usepackage{subfig}
\newcommand{\noun}[1]{\textsc{#1}}

\theoremstyle{plain}
\newtheorem{theorem}{Theorem}
\newtheorem{proposition}{Proposition}
\newtheorem{lemma}{Lemma}
\newtheorem{corollary}{Corollary}
\theoremstyle{definition}

\theoremstyle{remark}
\newtheorem{remark}{Remark}

\newcommand{\lcm}[1]{\operatorname{lcm}\left(#1 \right)}

\def\RR{\mathbb{R}}

\def\CC{\mathbb{C}}
\def\ZZ{\mathbb{Z}}
\def\NN{\mathbb{N}}

\newcommand{\bd}{\mathbf}
\newcommand{\Vgf}{\mathcal{V}_gf}
\renewcommand{\mod}{~\operatorname{mod}~}
\begin{document}

\begin{frontmatter}
\title{Efficient algorithms for discrete Gabor transforms on a nonseparable
lattice}

\author[nhg]{Christoph Wiesmeyr}
\ead{christoph.wiesmeyr@univie.ac.at}% <-this % stops a space
\author[ari]{Nicki Holighaus}
\author[ari]{Peter L. S{\o}ndergaard}

\address[nhg]{Numerical Harmonic Analysis Group, Faculty of Mathematics, 
University of Vienna, Austria}
\address[ari]{Acoustics Research Institute, 
Austrian Academy of Sciences, Vienna, Austria}

\begin{abstract}
The Discrete Gabor Transform (DGT) is the most commonly used 
transform for signal analysis and synthesis using a linear frequency
scale. It turns out that the involved operators are rich in structure
if one samples the discrete phase space on a subgroup. Most of the
literature focuses on separable subgroups, in this paper we will
survey existing methods for a generalization to arbitrary groups, as
well as present an improvement on existing methods. Comparisons are made with
respect to the computational complexity, and the running time of
optimized implementations in the C programming language. The new
algorithms have the lowest known computational complexity for
nonseparable lattices and the implementations are freely available
for download. By summarizing general background information on the
state of the art, this article can also be seen as a research survey,
sharing with the readers experience in the numerical work in Gabor
analysis.
\end{abstract}


\end{frontmatter}

\section{Introduction}

Over the past 20 years the \emph{Gabor transform} has become a very valuable
and widely used tool in signal processing. The
finite, discrete \emph{Short time Fourier transform} (STFT) for a given
signal $f$ of length $L$ is computed by testing $f$ against
shifted and modulated copies of a window function $g$
\begin{equation*}
  \mathcal V_g f(x,\omega)=\sum_{l=0}^{L-1} f(l)\overline{g(l-x)}e^{-2 \pi i \omega l/L}.
\end{equation*}
The Gabor transform is a sampled version of the STFT and both provide
the possibility to extract temporal frequency information from the signal.
The space spanned by the two variables $x, \omega$ is called the \emph{time-frequency
plane}; more precise information can be found in Section \ref{sec:prelims}. A family
of translations and modulations of a window function is called \emph{Gabor family} 
or \emph{Gabor system}.

There exist a continuous counterparts of the STFT and the Gabor transform.
The time frequency plane is $\RR^2$ in this case and 
general sampling sets have received attention, e.g. \citep{fegr92-1}.
Sampling this plane on a discrete subgroup, also called \emph{lattice}
\citep{gr01} admits rich structure, as described in the next section.
It has recently been conjectured
that for a standard Gaussian window the best sampling strategy is
a regular hexagonal pattern \citep{abdo12}. Geometric arguments have also
lead to the same sampling strategies in undersampled systems for
pulse shape design in wireless channel estimation \citep{best03}.

In the discrete setting efficient algorithms exist almost exclusively for
sampling on separable or rectangular lattices \citep{ltfatnote015},
the most important ones discovered quickly after finding the Fast
Fourier transform algorithm \citep{cooley1965algorithm}. The two approaches that
are most commonly used are the {\em overlap-add algorithm},
\citep{helms1967fast, stockham1966high} and the {\em weighted
  overlap-add algorithm} \citep{schafer1973design,po76}. Both of these
algorithms require that the window is {\em Finite Impulse Response}
(FIR), i.e. the size of its support is much smaller than its
length. Fast, but less well known algorithms without
this requirement have also been found \citep{bage96,ltfatnote011}.

It is a natural question how to generalize existing algorithms for the Gabor transform
and its inverse to the case of nonseparable lattices. 
In the early years of this century there has been a series of papers and investigations
on this subject \citep{bastiaans1998rectangular, bastiaans1998modified,van2000gabor,bastiaans2001gabor,bastiaans2003gabor} 
and more by the same authors, also
collected in~\citep{va01-2}. Earlier studies focus on the computation of dual Gabor windows
on nonseparable lattices, using iterative methods~\citep{fekapr97} or harnessing the
block structure of Gabor analysis and frame operators directly and reducing
nonseparable sampling sets to a union of product lattices~\citep{pr96,fekoprst96}. 
Another contribution came some years later further investigating 
the discrete theory of \emph{metaplectic
operators} \citep{fehakamane08}. In this paper we present
approaches from these works and propose an improved algorithm,
which allows for more efficient computation.

There are two fundamentally different ways of realizing computations
that we will investigate and improve upon. The first one uses a
decomposition of a nonseparable lattice into the union of co-sets of a
sparser separable lattice similar to \citep{fekoprst96,zezi97,bava04,va01-2}. 
This will allow to
write the Gabor family as a union of Gabor families on this sparse
lattice with different windows.  We call such a system \emph{multiwindow
Gabor} family, since it shares much of the structure from standard
Gabor systems \citep{zezi93}. The details can be found in Subsection
\ref{ssec:multiwin}.

The second method under consideration uses the fact that any lattice
can be written as the image of a rectangular lattice under an
invertible lattice transform. For a special subset of these
transforms, so called \emph{symplectic operators} on the signal space
exist that allow to reduce all the computations for Gabor
systems on nonseparable lattices to Gabor systems on rectangular
lattices. It turns out that in the $1$ dimensional setting the
transformation to the separable case is always possible
\citep{fehakamane08,kane09}. This method has first been described for
the continuous case, a summary can be found in \citep{gr01} and then
translated into the finite discrete setting, where it takes more
effort to obtain the results due to number theoretic
considerations. The algorithms presented in Subsection \ref{ssec:snf}
are based on the results in \citep{fehakamane08} and improved in 
\ref{ssec:shears}.
    
In higher dimensions the class of lattices that can be reduced to
a rectangular sampling strategy is expected to be a strict subset of all
lattices. While the class of lattice transforms that admit a
symplectic operator, called \emph{symplectic matrices}, can be
determined explicitly it is not easy to see whether a given lattice
can be transformed to rectangular shape using this class of matrices.
In contrast to the difficulties with generalizing the metaplectic
approach to higher dimensions, the multiwindow decomposition
can be extended directly. However, the description of the multidimensional 
case is beyond the scope of this contribution.

After introducing the necessary basic concepts in Section \ref{sec:prelims}, 
 we mainly present the different approaches in Section \ref{sec:nonseplatts}.
 Section \ref{sec:Implementation} describes the implementation
of the different algorithms and compares their computational complexity and running time.

\section{Preliminaries}\label{sec:prelims}
We use the ``$\cdot$'' notation in conjunction with the DFT to denote
the variable over which the transform is to be applied.

\subsection{Gabor frames on subgroups of the TF-plane}

  We recall some basics from Gabor analysis, frame theory and the theory of metaplectic operators on $\CC^L$. A \emph{Gabor system} in $\CC^L$ is a set of functions of the form
  \begin{equation}\label{eq:GabSys1}
    \mathcal{G}(g,\Lambda) := \{ \bd{M}_{\omega} \bd{T}_x g ~:~ (x,\omega)^T\in\Lambda\subseteq \ZZ_L^2 \},
  \end{equation}
  where $g\in \CC^L$ and $\mathbf{T}_x$, $\mathbf{M}_{\omega}$ denote a time shift by $x$ and a frequency shift (or modulation) by $\omega$, i.e.
  \begin{equation*}
    \mathbf{T}_x f(l) = f(l-x) \quad \text{and} \quad \mathbf{M}_\omega f(l) = e^{2\pi i l\cdot \omega/L}f(l),
  \end{equation*}
  with $l-x$ considered modulo $L$. Thus, a Gabor system is a set of time-frequency shifts of a fixed function $g$. For some given $x$ and $\omega$ we introduce also the notation of a time-frequency shift operator
  \begin{equation*}
   \pi(x,\omega) = \bd{M}_{\omega} \bd{T}_x.
  \end{equation*}

  The Gabor coefficients of some $f\in \CC^L$, with respect to $\mathcal{G}(g,\Lambda)$ are given by the samples of the Short-time Fourier transform 
  \begin{equation*}
    \Vgf (x,\omega) = \langle f, \bd{M}_{\omega} \bd{T}_x g \rangle = \sum_{l=0}^{L-1} f(l)\overline{g(l-x)}e^{-2\pi i\omega l/L}, 
  \end{equation*}
  for $(x,\omega)^T\in \Lambda$.
  
  It is important to know if the signal $f$ can be reconstructed from its transform coefficients $\{c_{x,\omega}=\Vgf(x,\omega)\}_{(x,\omega)^T\in\Lambda}$. If so, we call the Gabor system $\mathcal{G}(g,\Lambda)$ a frame. It turns out that this is equivalent to the invertibility of the so-called frame operator defined as
  \begin{equation}\label{eq:frameop} 
    \mathbf{S}_{g,\Lambda}f = \sum_{(x,\omega)^T\in \Lambda} \langle f, \pi(x,\omega)g \rangle \pi(x,\omega)g. 
  \end{equation} 
  From here on, we will use the shorthand notation $\mathbf S = \mathbf{S}_{g,\Lambda}$ whenever there is no confusion as to the Gabor system $\mathcal{G}(g,\Lambda)$ used. By inversion of this operator we can give an explicit inversion formula
  \begin{equation*}
   f =   \sum_{(x,\omega)^T\in \Lambda} c_{x,\omega} \mathbf{S}^{-1}\pi(x,\omega)g.
  \end{equation*} 
  The family $\{\mathbf{S}^{-1}\pi(x,\omega)g\}_{(x,\omega)^T\in\Lambda}$ is called the (canonical) dual Gabor system. If $\Lambda$ is a subgroup of the phase space, then we know from standard Gabor theory that the dual system is a Gabor system itself, given by $\mathcal G (\mathbf S^{-1}g,\Lambda)$ , see e.g. \citep{gr01}.
  In the following we will only consider this structured case and denote the subgroup relation by $\Lambda \leq \ZZ_L^2$.

  It is easy to see that for any matrix $A \in \ZZ_L^{2\times 2}$ the set $A \ZZ_L^2$ forms a subgroup of the time-frequency plane. The following proposition  shows that the converse is also true.
  Furthermore, it suggests a normal form that allows us to establish a one to one relation between lattices and generating matrices. Further implications of this bijection can be found in \citep{hahotowi12}.
  \begin{proposition}\label{pro:lattNF}
    For every $\Lambda \leq \ZZ_L^2$ there exist unique $a,b|L$, $0\leq s < b$ and $s\in\frac{ab}{\gcd(ab,L)}\ZZ$, such that 
    \begin{equation}\label{eq:lattNF}
      \Lambda = A\ZZ_L^2 = \left(\begin{array}{cc} a & 0 \\ s & b \end{array}\right)\ZZ_L^2.
    \end{equation}
  \end{proposition}
    \begin{proof}
      Existence: For $\Lambda$ to be a subgroup of $\ZZ^2_L$, $B:= \{\omega\in\ZZ_L ~:~ (0,\omega)\in\Lambda\} \leq \ZZ_L$ must hold. Set $b = \min(B)$ and $a = \min\{x\in\ZZ_L ~:~ \{(x,\omega)^T\in\Lambda\}\neq 0\}$, then $a,b|L$ and the cardinality of $\Lambda$ is $|\Lambda| = L^2/(ab)$, i.e. $\Lambda$ has $L/a$ equidistant nonempty columns, with $L/b$ equidistant elements each. Finally, set $s = \min\{\omega\in\ZZ_L ~:~ (a,\omega)\in\Lambda\}$, then $0\leq s<b$ follows easily. $s\in\frac{ab}{\gcd(ab,L)}\ZZ$ is obtained by observing that $sL/a \in b\ZZ$ must be fulfilled. Obviously, the linear span of $\{(a,s),(0,b)\}$ is contained in $\Lambda$ and of cardinality $L^2/(ab)$, hence equality holds.
      
      Uniqueness: Let $a,b,s$ be as constructed above. Since the cardinality of $\Lambda$ depends on the product $ab$, any change of $a$ implies a change of $b$. The condition $b|L$ guarantees $B = \{mb ~:~ m\in\ZZ_{L/b}\}\neq \{m\tilde{b} ~:~ b\neq\tilde{b}|L, m\in\ZZ_{L/\tilde{b}}\}$ determining $a,b|L$ uniquely. With $\{\omega\in\ZZ_L ~:~ (a,\omega)\in\Lambda\} = \{s+mb ~:~ m\in\ZZ_{L/b}\}$, we see that $\tilde{s}=s+mb \geq b$ if and only if $m\neq 0$.
    \end{proof}
  
  With $a,b,s$ as in \eqref{eq:lattNF}, we define 
  \[ \begin{split} \lefteqn{\mathcal{G}(g,a,b,s):=\mathcal{G}(g,\Lambda)} & \\ & = \{ g_{n,k} := \bd{M}_{sn+bk} \bd{T}_{an} g ~:~ (n,k)\in\ZZ_{L/a}\times\ZZ_{L/b} \}\end{split} \]
  for $\Lambda = A\ZZ_L^2$, omitting $s$ if it equals zero. Lattices with $s=0$ are called \emph{separable, rectangular} or \emph{product lattices}, since they can be written as the direct product of two subgroups of $\ZZ_L$. If $s\neq 0$, we call a lattice \emph{nonseparable}. It is easy to see that the unique lower triangular form can be rewritten into an upper triangular matrix.

  \begin{proposition} \label{pro:equilatt}
  Given a subgroup $\ZZ_L^2$ in normal form, i.e. given $a$, $b$ and $s$. Then the following representations are equivalent
  \begin{equation*}
    \begin{pmatrix}
      a & 0 \\
      s & b
    \end{pmatrix} \cdot \ZZ_L^2 = 
    \begin{pmatrix}
      \tilde a & \tilde s \\
      0 & \tilde b
    \end{pmatrix} \cdot \ZZ_L^2,
  \end{equation*}
  where $\tilde b=\gcd(b,s)$, $\tilde a = ab/\gcd(b,s)$. Furthermore, we use B\'ezout's identity to represent $k_1 s + k_2 b = \gcd(b,s)$, then $\tilde s = k_1 a$. 
  \end{proposition}  
  \begin{proof}
  By computation one can verify that
  \begin{equation*}
    \begin{pmatrix}
      a & 0 \\
      s & b
    \end{pmatrix} \cdot
    \begin{pmatrix}
      b/\gcd(b,s) & k_1 \\
      -s/\gcd(b,s) & k_2
    \end{pmatrix}=
    \begin{pmatrix}
	\tilde a & \tilde s \\
	0 & \tilde b
    \end{pmatrix}
  \end{equation*}
  The second matrix has determinant $1$ and therefore is invertible. The assertion follows because $Q\cdot \ZZ_L^2 = \ZZ_L^2$ for any invertible matrix.
  \end{proof}  
  
In some cases we will switch to another description of a subgroup as it comes up more natural in some settings. Instead of the shear parameter $s$, one can also use the shear relative to $b$, given by
\begin{equation*}
    \lambda = \frac{s}{b} = \frac{\lambda_1}{\lambda_2}, \text{ with } \lambda_1 = \frac{s}{\gcd(b,s)},\ \lambda_2 = \frac{b}{\gcd(b,s)},
\end{equation*}
This easily explains how to convert $s$ into $\lambda_1$ and $\lambda_2$ and vice versa. A visualization can be found in Figure \ref{fig:latticetypes}. Unlike in the case of separable lattices, there is no immediate natural way of indexing the Gabor coefficients. However, it seems sensible to index by the position in time and counting the sampling points in frequency from the lowest nonnegative frequency upwards. Therefore we will fix
\begin{equation}
  c\left(m,n\right)=\sum_{l=0}^{L-1}f(l)\overline{g(l-an+1)}e^{-2\pi il(m+w(n))/M},
  \label{eq:DGT}
\end{equation}
for the rest of this contribution, where the additional offset $w$ is given by $w(n)=\mod(n\lambda_1,\lambda_2)/\lambda_2$. This format is also implemented in the open source \noun{MATLAB}/\noun{Octave} Toolbox \emph{LTFAT}~\citep{ltfatweb}, used for the experiments in Section \ref{sec:Implementation}.

\begin{figure}[t!h]

\subfloat[\vspace{-3pt}$\lambda_{1}/\lambda_{2}=0$]
{\includegraphics[width=0.40\textwidth,trim=0 20 0 0, clip]{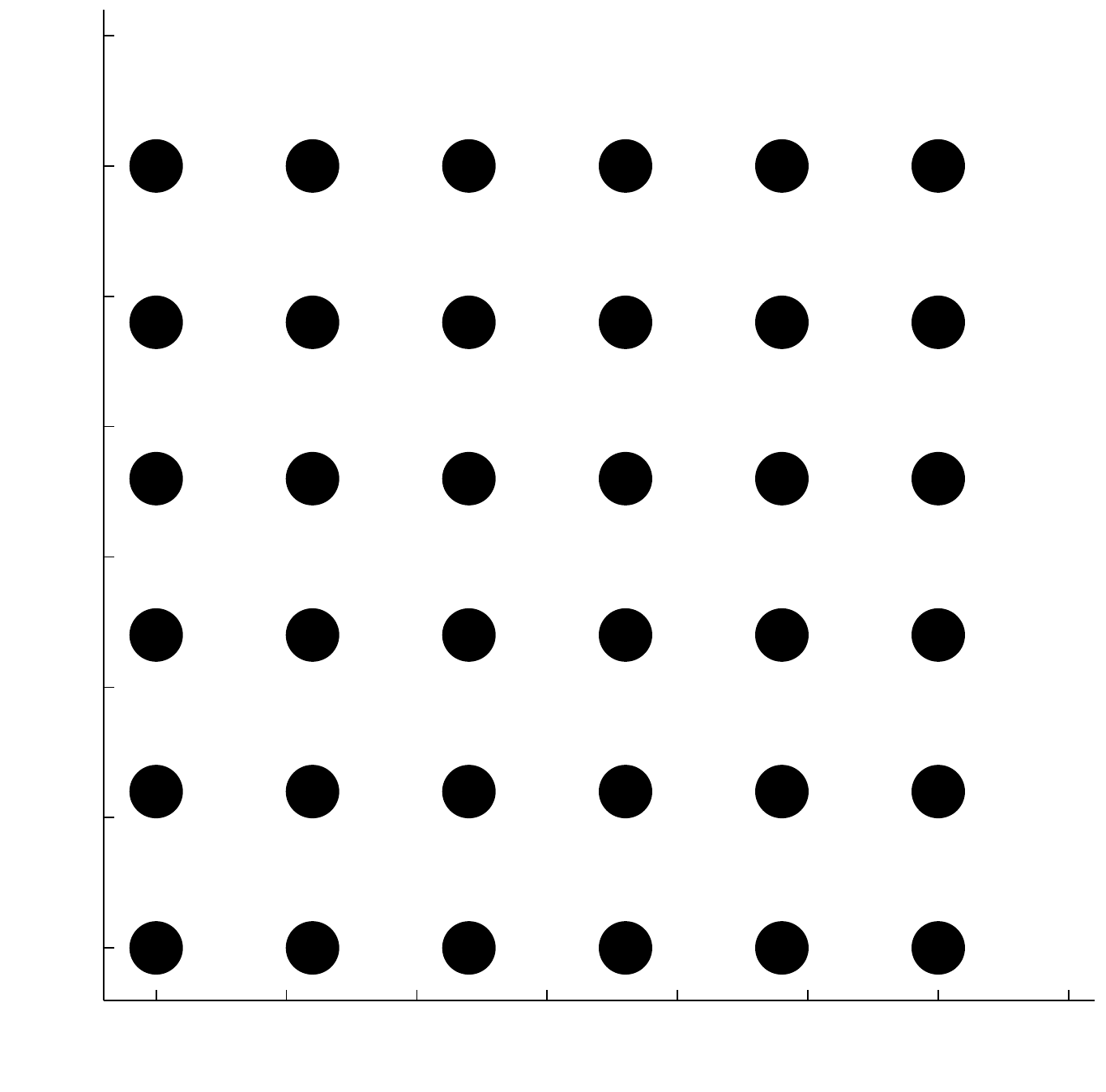}}
{\hfill}\subfloat[\vspace{-3pt}$\lambda_{1}/\lambda_{2}=1/2$]
{\includegraphics[width=0.40\textwidth,trim=0 20 0 0, clip]{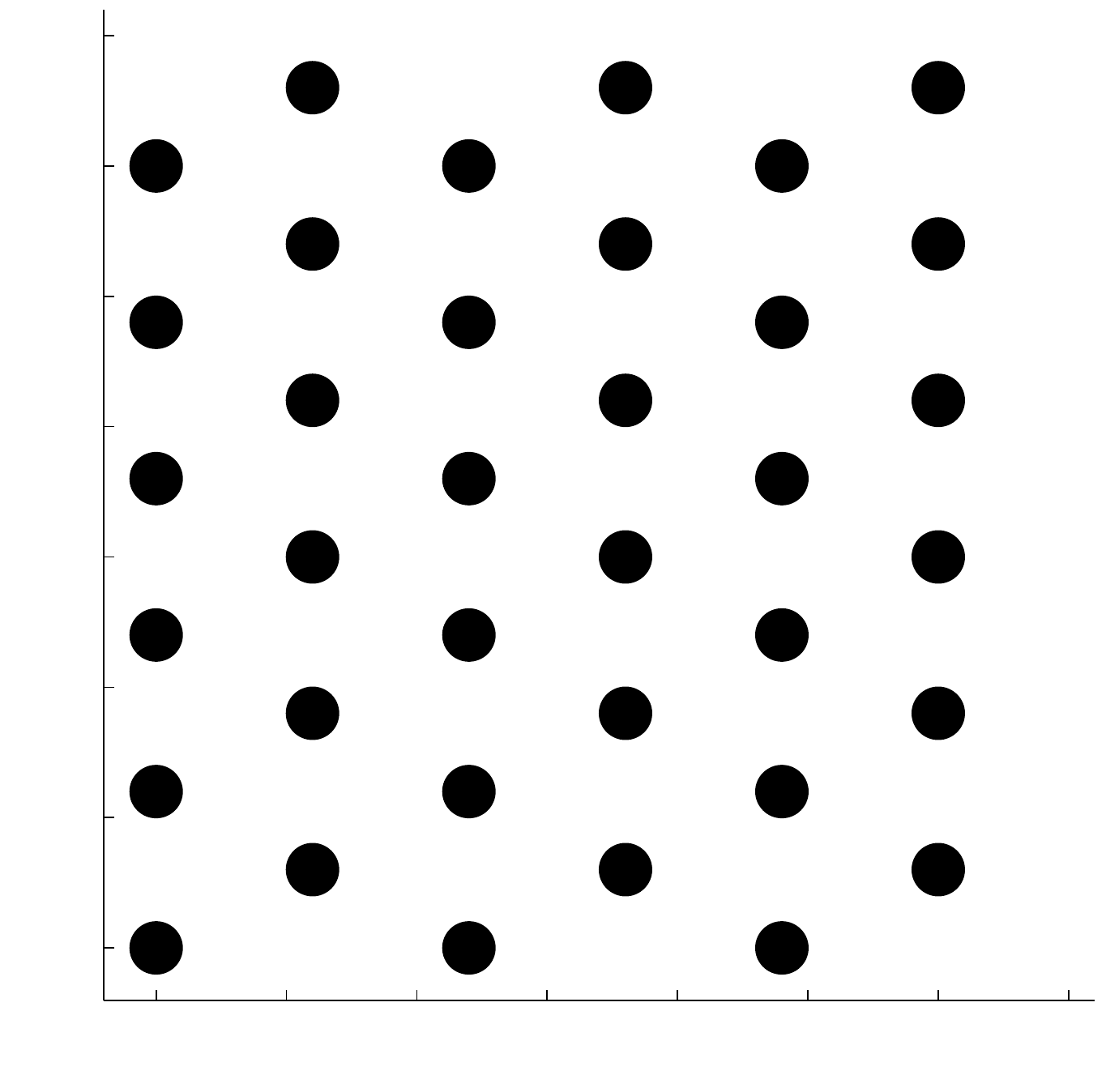}}

\subfloat[\vspace{-3pt}$\lambda_{1}/\lambda_{2}=1/3$]
{\includegraphics[width=0.40\textwidth,trim=0 20 0 0, clip]{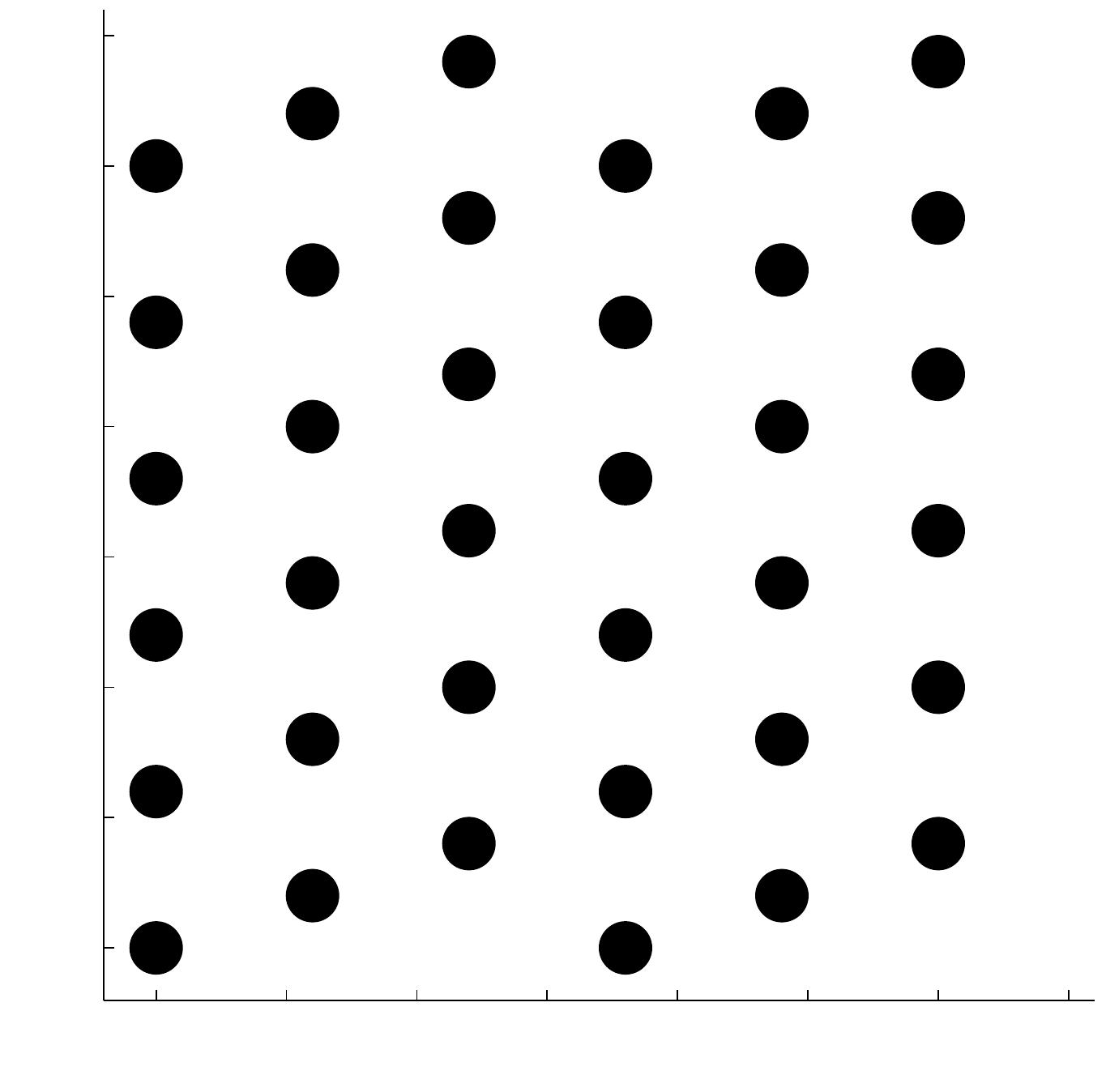}}
{\hfill}\subfloat[\vspace{-3pt}$\lambda_{1}/\lambda_{2}=2/3$]
{\includegraphics[width=0.40\textwidth,trim=0 20 0 0, clip]{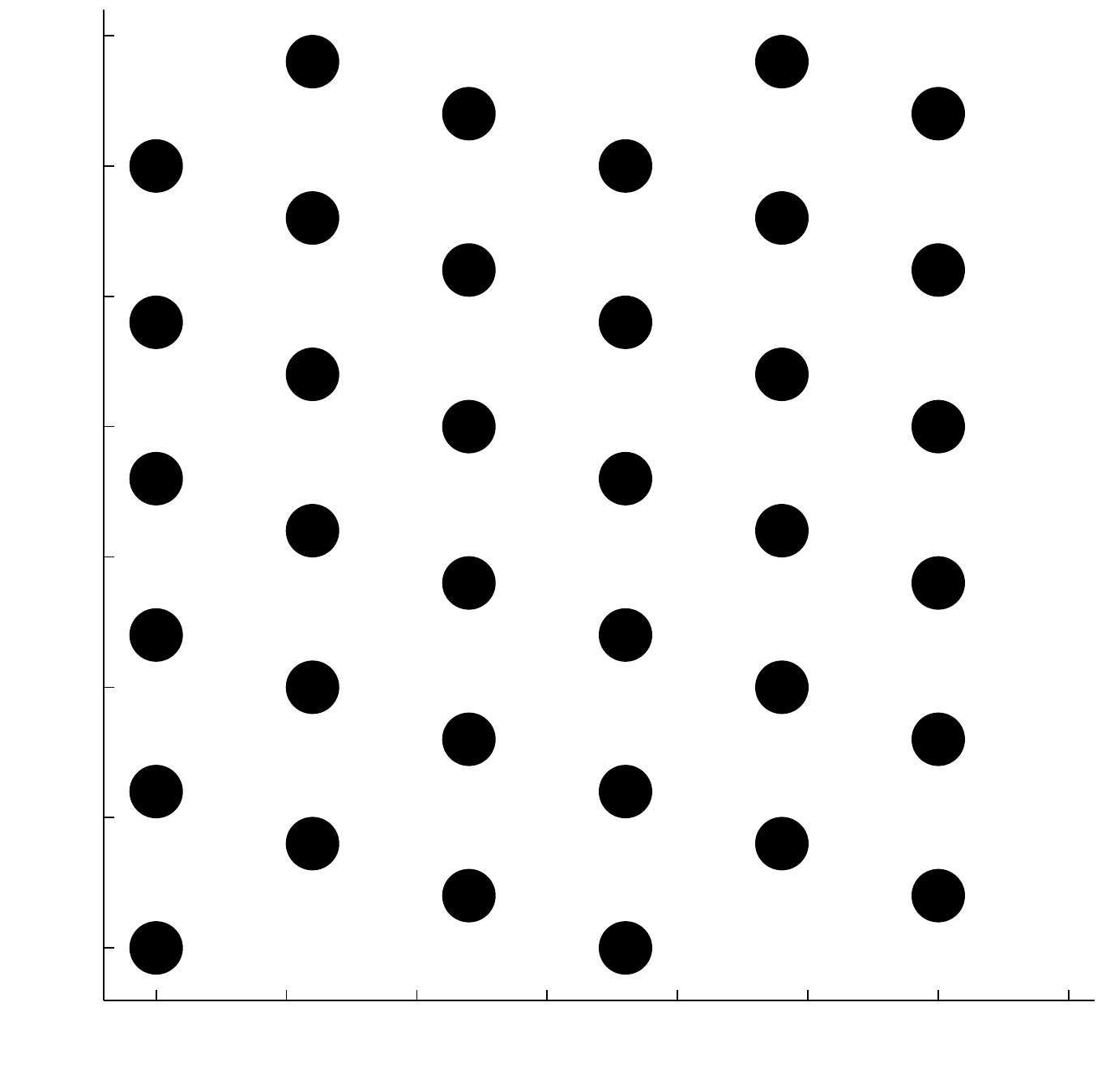}}

\caption{\label{fig:latticetypes}The figure shows the placement of
  the Gabor atoms for four different lattice types in the
  time-frequency plane . The displayed Gabor system has parameters $a=6$, $M=6$ and
  $L=36$. The lattice (a) is  called {\em rectangular} or
  {\em separable} and the lattice (b) is known as the {\em quincunx} lattice.}

\end{figure}

\subsection{Metaplectics}\label{ssec:metaplecs}
  
A metaplectic operator, loosely speaking, is the signal domain counterpart to a symplectic transform of the lattice on phase space. A comprehensive treatment of these operators in the finite discrete setting can be found in \citep{kane09}. In this contribution we will be focusing on the one dimensional setting, for which the operators are described in detail in \citep{fehakamane08}. In this section we will formulate some results that will prove to be important in subsequent sections. We start by the factorization of a lattice generator into elementary matrices, which we will denote by
\begin{equation}\label{eq:elem_sympl}
  \begin{split}
  \lefteqn{F = \left(\begin{array}{cc} 0 & -1 \\ 1 & 0 \end{array}\right), \quad S_c = \left(\begin{array}{cc} 1 & 0 \\ c & 1     \end{array}\right),} \hspace{40pt}& \\  
	    & D_a = \left(\begin{array}{cc}
	      a & 0 \\ 0 & a^{-1}
	    \end{array}\right), \hspace{40pt}
  \end{split}
\end{equation}
where $c\in \ZZ_L$ and $a \in \ZZ_L$ invertible.

\begin{proposition}[Feichtinger et al. (2008)~\citep{fehakamane08}]\label{pro:weyldecompose}
  Let $M = \left(\begin{smallmatrix} a & b \\ c & d\end{smallmatrix}\right)\in\ZZ_L^2$ with $\det(M) = 1$, then there exists $m\in\ZZ$ such that $a_0 = a + mb$ is invertible in $\ZZ_L$. Let $c_0 = c +md$, then 
  \begin{equation*}
    M = S_{c_0a_0^{-1}}D_{a_0}F^{-1}S_{-a_0^{-1}b}FS_{-m}.
  \end{equation*}
\end{proposition}

The proof is based on Weil's decomposition of arbitrary symplectic matrices into a composition of elementary symplectic matrices as in \eqref{eq:elem_sympl}.

\begin{lemma}\label{lem:meta_phases}
 For the above defined matrices we define the corresponding metaplectic operators as follows
 \begin{align*}
  F &\mapsto \mathbf U_F = \mathcal F \\
  S_c &\mapsto \mathbf U_{S_c} = \left( f(\cdot) \mapsto f(\cdot)\exp(\pi i c \cdot^2 (L+1)/L) \right) \\
  D_a &\mapsto \mathbf U_{D_a} = \left( f(\cdot) \mapsto f(a^{-1}\cdot) \right).
 \end{align*}
 With these transformations the following hold for all $\lambda \in \ZZ_L^2$
 \begin{align*}
  \mathbf U_F \pi(\lambda) &= \phi_F(\lambda)\pi(F\lambda) \mathbf U_F \\
  \mathbf U_{S_c} \pi(\lambda) &= \phi_{S_c}(\lambda)\pi(S_c \lambda) \mathbf U_{S_c} \\
  \mathbf U_{D_a} \pi(\lambda) &= \phi_{D_a}(\lambda)\pi(D_a\lambda) \mathbf U_{D_a},
 \end{align*}
 where $\phi_F$, $\phi_{S_c}$ and $\phi_{D_a}$ are phase factors.
\begin{proof}
  Some simple calculations are sufficient to establish the result:
  \begin{align*}
    \mathbf U_F \pi(\lambda)f & = \mathcal F \bd{M}_\omega\bd{T}_xf = T_\omega M_{-x} \hat f \\
		       & = e^{-2\pi ix\omega/L}\bd{M}_{-x}\bd{T}_\omega \hat f = e^{-2\pi ix\omega/L}\pi(F\lambda)\mathbf U_Ff,
  \end{align*}
  \begin{align*}
    \mathbf U_{S_c} \pi(\lambda)f & = \bd{M}_\omega\bd{T}_x e^{\pi ic(\cdot + x)^2(L+1)/L}f\\
		       & = e^{\pi icx^2(L+1)/L}\bd{M}_{\omega+cx}\bd{T}_x\mathbf U_{S_c}f \\
		       & = e^{\pi icx^2(L+1)/L}\pi(S_c\lambda)\mathbf U_{S_c}f
  \end{align*}
  and
  \begin{align*}
    \mathbf U_{D_a} \pi(\lambda)f & = \bd{M}_{a^{-1}\omega}f(a\cdot - x) = \bd{M}_{a^{-1}\omega}f\left(a^{-1}(\cdot - ax)\right)\\
		       & = \bd{M}_{a^{-1}\omega}\bd{T}_{ax} \mathbf U_{D_a}f = \pi(D_a\lambda)\mathbf U_{D_a}f.
  \end{align*} 
\end{proof}
\end{lemma}

The combination of the two results above immediately yields the following theorem.

\begin{theorem}\label{thm:meta_shifts}
 For any matrix $M\in \ZZ_L^2$ with $\det (M)=1$, there exists a metaplectic operator $\mathbf U_M$, such that for all $\lambda \in \ZZ_L^2$
 \begin{equation*}
  \mathbf U_M \pi(\lambda) = \phi_M(\lambda)\pi(M\lambda) \mathbf U_M.
 \end{equation*}
\end{theorem}

\section{Computation on nonseparable lattices}\label{sec:nonseplatts}

Nonseparable lattices in $\ZZ_L^2$ can be interpreted in a variety of
ways. Several different approaches relate Gabor expansions on general
lattices to one or several equivalent expansions on separable (or
rectangular) lattices. From an algorithmic viewpoint, these are of
particular interest, since a wealth of research
\citep{po76,allen1977unified,auslander1991discrete,zezi93,bage96,st98-8}
has investigated efficient algorithms for analysis and synthesis using
Gabor dictionaries on separable lattices. Each of the three approaches
described in this section yields a simple relation between arbitrary
given Gabor systems and Gabor systems on separable sampling sets that
can be harnessed for efficient analysis and synthesis.
  
  \subsection{Correspondence via multiwindow Gabor}\label{ssec:multiwin}
  
   We will decompose a given lattice into a union of co-sets
  of a sparser separable lattice, which will allow us to use
  multiwindow methods \citep{zezi96, zezi97, zezi97-1, zezi98} for the computation. 
  Using multiwindow methods for computation of Gabor transforms on nonseparable lattices 
  has been proposed in \citep{fekoprst96, zezi97} and implementation has been discussed in 
  \citep{bava04,va01-2}. However, the latter only briefly mention the 
  computation of dual Gabor windows, not discussing efficient implementation in detail.
  
  \begin{proposition} \label{pr:mwindec}
   Given the lattice $\Lambda$ in normal form specified by the parameters $a,b$ and $s$, then
   \begin{equation*}
    \Lambda = \cup_{m=0}^{\lambda_2-1} \left((am,sm \mod b)^T + \tilde \Lambda\right),
   \end{equation*}
   where $\lambda_2=b/\gcd(b,s)$ and $\tilde{\Lambda}$ is the separable lattice generated by $(\lambda_2 a,0)^T$ and $(0,b)$.
  \end{proposition}

  \begin{proof}
    Let the matrix generating $\Lambda$ be denoted by $A$ and let us define $M_x = \left\{ sx+b\omega:\; \omega \in \ZZ_L \right\}$, for $0 \leq x < L/a$. We note here, that $M_x$ is the second coordinate of the set $ A\cdot (x,\ZZ_L)^T$.
    Furthermore, $0 \in M_x$ if and only if $x$ is a multiple of $\lambda_2$.
    To see that, we first note that $\lambda_1$ and $\lambda_2$ are relatively prime.
    Then the following equation has a solution if and only if $x$ is a multiple of $\lambda_2$
    \begin{equation*}
      sx+b\omega = b \left( \frac{\lambda_1}{\lambda_2} x + \omega \right)=0.
    \end{equation*}
    This yields
   \begin{align*}
    M_x &= M_{x+\lambda_2}, \quad \text{for $x \in \ZZ_{L/a}$}\\
    M_x &= sx\mod b + M_0 
   \end{align*}
   This observation yields the following decomposition of the original lattice
   \begin{equation*}
    \begin{split}
     \Lambda &= \bigcup_{x=0}^{L/a-1} \{ ax \} \times M_x \\
	  &= \bigcup_{m=0}^{\lambda_2-1} \bigg((am,sm \mod b)^T \\ & \hspace{40pt} 
	  + \bigcup_{j=0}^{L/(a\lambda_2)-1} \{ aj\lambda_2 \} \times M_0 \bigg),
    \end{split}
   \end{equation*}
   which finishes the proof by observing
   \begin{equation*}
    \tilde \Lambda = \bigcup_{j=0}^{L/(a\lambda_2)-1} \{ aj \lambda_2 \} \times M_0.
   \end{equation*}
  \end{proof}

    We can now describe a Gabor system $\mathcal{G}(g,\Lambda)$, with $\Lambda$ in the form \eqref{eq:lattNF}, and the related operators completely in terms of a union of Gabor systems $\mathcal{G}(g_m,\tilde{\Lambda})$ on the separable lattice $\tilde{\Lambda}$.
    
    \begin{proposition}
      Let $\mathcal{G}(g,\Lambda)$, $\mathcal{G}(g_m,\tilde{\Lambda})$, with $\Lambda, \tilde{\Lambda}$ as in Proposition \ref{pr:mwindec} and $g\in\CC^L$, $g_m = \bd{M}_{ms \mod b}\bd{T}_{ma}g$, for $0\leq m < \lambda_2$, be Gabor systems,   then 
      \begin{equation}\label{eq:mw_frop}
        \bd{S}_{g,\Lambda}f  = \sum_{m=0}^{\lambda_2-1}\bd{S}_{g_m,\tilde{\Lambda}}f.
      \end{equation}
      Moreover, the Gabor transform can be computed using the identity
      \begin{equation}\label{eq:mw_phases}
      \begin{split}
        \lefteqn{\langle f, \bd{M}_{kb+(ms\mod b)} \bd{T}_{na}g \rangle} \\
        & = e^{-2\pi i\tilde{n}\tilde{a}(ms \mod b)/L}\langle f, \bd{M}_{kb}\bd{T}_{\tilde{n}\tilde{a}}g_m \rangle,
      \end{split}
      \end{equation}
      where $\tilde{n} = \lfloor n/\lambda_2 \rfloor$ and $m = n-\tilde{n}$.
      \begin{proof}
        Analogous to Lemma \ref{lem:meta_phases}, we find that
        \begin{align*}
          \lefteqn{\bd{M}_{kb+(ms\mod b)} \bd{T}_{na}g} \\
                                               & = \bd{M}_{kb}\bd{M}_{ms\mod b} \bd{T}_{\tilde{n}\tilde{a}}\bd{T}_{ma}g \\
					       & = e^{2\pi i\tilde{n}\tilde{a}(ms \mod b)/L}\bd{M}_{kb}\bd{T}_{\tilde{n}\tilde{a}}\bd{M}_{ms\mod b}\bd{T}_{ma}g \\
					       & = e^{2\pi i\tilde{n}\tilde{a}(ms \mod b)/L}\bd{M}_{kb}\bd{T}_{\tilde{n}\tilde{a}}g_m,
        \end{align*}
        yielding \eqref{eq:mw_phases}. Using $kb+(ms\mod b) = kb+(ns\mod b) = (k-\lfloor ns/b \rfloor)b+ns$, since $\tilde{n}s\mod b = 0$ allows to derive \eqref{eq:mw_frop} by the identity
        \begin{equation*}
         \begin{split}
         \lefteqn{\sum_{n=0}^{L/a-1}\sum_{k=0}^{L/b-1} \langle f,\bd{M}_{ns+kb} \bd{T}_{na}g \rangle \bd{M}_{ns+kb} \bd{T}_{na}g}\\
         & = \sum_{m=0}^{\lambda_2-1} \sum_{\tilde{n}=0}^{L/\tilde{a}-1}\sum_{k=0}^{L/b-1} \langle f,\bd{M}_{kb}\bd{T}_{\tilde{n}\tilde{a}}g_m \rangle \bd{M}_{kb}\bd{T}_{\tilde{n}\tilde{a}}g_m.
         \end{split}
        \end{equation*} 
      \end{proof}
    \end{proposition}

  \subsection{Correspondence via Smith normal form}\label{ssec:snf}
    In this and the following section, we aim to describe an arbitrary lattice as separable lattice under a symplectic deformation, i.e. we will determine a symplectic matrix $P$, such that $\Lambda = P\tilde{\Lambda}$ for a general lattice $\Lambda$ and a separable lattice $\tilde{\Lambda}$. This problem is equivalent to decomposing the lattice generator matrix $A\in\ZZ_L^{2\times 2}$ into $A = PDV$, with a diagonal matrix $D$, a determinant $1$ matrix $V$ and a symplectic matrix $P$. We observed earlier that any determinant $1$ matrix in $\ZZ_L^{2\times 2}$ is symplectic. Thus, this decomposition is accomplished by applying Smith's algorithm for matrices in $\ZZ^{2\times 2}$ to determine the Smith normal form $\tilde{D}$ of $A$ and transformation matrices $\tilde{P},\tilde{V}$, followed by considering the entries of $\tilde{D},\tilde{P},\tilde{V}$ modulo $L$ to find $D,P,V$.
    
    The following Proposition by Feichtinger et al. was originally published in \citep{fehakamane08}, where the proof is also presented. The procedure of computing Gabor transforms and dual windows using the methods in this section have been proposed therein, but their implementation was not discussed in detail. 
    
    \begin{proposition}
     Let $\Lambda = A\ZZ_L^2$ be a lattice and $A=\tilde{P}\tilde{D}\tilde{V}$ the Smith decomposition of $A$. Then
     \begin{equation*}
      \Lambda = P \tilde \Lambda,
     \end{equation*} 
     where $P = (\tilde P \mod L)$, $D = (\tilde D \mod L)$ and $\tilde \Lambda = D \ZZ_L^2$.
    \end{proposition}
    
    Using Proposition \ref{pro:weyldecompose} and Lemma \ref{lem:meta_phases} one obtains the operator $\mathbf U_P$ corresponding to the symplectic matrix $P$ and this leads to the final computational procedure described in the following Corollary.
    
    \begin{corollary}\label{pro:metaform}
     Let the notation be as in the previous proposition. Then one finds for the symplectic matrix $P$ and the corresponding metaplectic operator $\mathbf U_P$ by setting $\tilde{g} = \bd{U}_{P}^{-1}g$
     \begin{equation*}
      \bd{S}_{g,\Lambda} = \bd{U}_P\bd{S}_{\tilde g, \tilde{\Lambda}}\bd{U}_P^{-1}.
     \end{equation*} 
     Furthermore, the Gabor coefficients can be computed using the identity
     \begin{equation*}
      \langle f, \pi\left(z\right)g \rangle = \phi_P(z)\langle \bd{U}_P^{-1}f, \pi\left(P^{-1}z\right)\tilde{g} \rangle,
     \end{equation*} 
      for all $z = (x,\omega)^T \in\Lambda$.
    \end{corollary}

  \subsection{Correspondence via shearing}\label{ssec:shears}
  
As detailed in the previous section, the Weil decomposition and Smith normal form can be used to show that any lattice in $\ZZ^2_L$ can be written as a separable lattice, deformed by $6$ elementary symplectic matrices. This number can be reduced to $4$ or less as shown in the following theorem, which we will prove at the end of this section. Reducing computations on nonseparable lattices to the product lattice case via a shear operation has been proposed earlier \citep{bava04, va01-2}, however the authors were able to describe only a subset of all lattices over $\ZZ^2_L$ as shears of rectangular lattices. In \citep{va01-2} the author speculates that it might be possible to describe every lattice a the image of a product lattice under a horizontal and a vertical shear. In this section, we prove that this is indeed possible.

The proper definition of discrete, finite chirps, necessary to perform time-frequency shearing, has been a matter of some discussion, see e.g. \citep{cafi06}. While the naive linear chirp $\exp(2\pi ist^2/L)$ is still used by Bastiaans and van Leest~\citep{bava04, va01-2}, a more appropriate definition, see Lemma \ref{lem:meta_phases}, has been proposed by Kaiblinger \citep{ka99-1,fehakamane08}, constituting a second degree character~\citep{we64}.

  \begin{theorem}\label{thm:shearform}
    let $A\in \ZZ_L^{2\times 2}$. There exist $s_0,s_1\in\ZZ_L$ and $V\in\ZZ_L^{2\times 2}$ with $|\det(V)| = 1$, such that
    \begin{equation}\label{eq:sheardecompose}
      A = U_{s_0,s_1}DV,
    \end{equation}
    where $D\in\ZZ_L^{2\times 2}$ is diagonal and 
    \begin{equation}\label{eq:sheardecompose2}
      U_{s_0,s_1} = S_{-s_1}F^{-1}S_{s_0}F %F^{-1}S_{-s_0}FS_{s_1}.
    \end{equation}
  \end{theorem}
  
  We can now rewrite Gabor transforms on nonseparable lattices in the vein of Proposition \ref{pro:metaform} using the metaplectic operator associated to $U_{s_0,s_1}$. Subsequently, we denote by $\bd{U}_{s_0,s_1}$ the metaplectic operator associated with $U_{s_0,s_1}$.
  
\begin{proposition}\label{pro:sheartrans}
  Let $\Lambda = A\ZZ_L^2$ be a lattice, $D,U_{s_0,s_1}$ as in the previous theorem and $\tilde{\Lambda} = D\ZZ_L^2$. Furthermore let $g\in\CC^L$ and $\tilde{g} = \bd{U}_{{s_0,s_1}}^{-1}g$.
  Then
  \begin{align}\label{eq:shearframeop}
    \bd{S}_{g,\Lambda}f & = \bd{U}_{{s_0,s_1}}\bd{S}_{\tilde g,\tilde{\Lambda}}\bd{U}_{{s_0,s_1}}^{-1}f
  \end{align}
  and
  \begin{equation}\label{eq:shearcoeffs}
    \begin{split}\lefteqn{\langle f, \bd{M}_{\omega}\bd{T}_x g \rangle} \\
    & = \phi_{U_{s_0,s_1}}(z)\langle \bd{U}_{{s_0,s_1}}^{-1}f, \bd{M}_{\omega - s1(x-s_0\omega)}\bd{T}_{x-s_0\omega}\tilde{g} \rangle,\end{split}
  \end{equation}
  for all $z = (x,\omega)^T$. Moreover,
  \begin{equation}\label{eq:shearcoeffs-2}
    \phi_{U_{s_0,s_1}}(z) = e^{\pi i (s_0\omega^2-s_1(x-s_0\omega)^2)(L+1)/L}.
  \end{equation}
\end{proposition}

\begin{proof}%[Proof of theorem \ref{pro:sheartrans}]
  Everything but the explicit form of the phase factor $\phi_{U_{s_0,s_1}}$ is a direct consequence of Lemma \ref{lem:meta_phases} and Theorem \ref{thm:shearform}, note 
  \begin{equation*}
   U_{s_0,s_1} = S_{-s_1}F^{-1}S_{s_0}F = \begin{pmatrix}
      1 & -s_0 \\
      -s_1 & s_0s_1+1
    \end{pmatrix}.
  \end{equation*}
  To complete the proof, set $y = (x-s_0\omega)$ and determine the phase factor explicitly: 
  \begin{align*}
    \lefteqn{\bd{U}_{S_{-s_1}}\mathcal{F}^{-1}\bd{U}_{S_{s_0}}\mathcal{F}\bd{M}_\omega\bd{T}_x f}\\
      & = e^{\pi i s_0\omega^2(L+1)/L}\bd{U}_{S_{-s_1}}\mathcal{F}^{-1}\bd{T}_\omega\bd{M}_{s_0\omega-x} \bd{U}_{S_{s_0}}\mathcal{F}f \\
      & = e^{\pi i s_0\omega^2(L+1)/L}\bd{U}_{S_{-s_1}}\bd{M}_\omega \bd{T}_{y} \mathcal{F}^{-1}\bd{U}_{S_{s_0}}\mathcal{F}f \\
      & = e^{\pi i (s_0\omega^2 - s_1 y^2)(L+1)/L}\bd{M}_{\omega-s_1y} \bd{T}_{y} \bd{U}_{S_{-s_1}}\mathcal{F}^{-1}\bd{U}_{S_{s_0}}\mathcal{F}f,
  \end{align*}
  where we used Lemma \ref{lem:meta_phases} and $\exp(2\pi i m(L+1)/L) = \exp(2\pi i m/L)$ for all $m\in\ZZ$.
\end{proof}

For Proposition \ref{pro:sheartrans} to be valid, it remains to prove Theorem \ref{thm:shearform}, establishing the representation of $A$ through $U_{s_0,s_1}$.

\begin{proof}[Proof of Theorem \ref{thm:shearform}]
  By Proposition \ref{pro:lattNF} we can assume without loss of generality that $A$ is in lattice normal form, i.e.
  \begin{equation*}
   A=
   \begin{pmatrix}
     a & 0 \\
     s & b
   \end{pmatrix}.
  \end{equation*}

  To prove equation \eqref{eq:sheardecompose}, we rewrite $U^{-1}_{s_0,s_1}A = DV$ with a diagonal matrix $D$ and a unitary matrix $V$. It can be seen that 
  \[
   U^{-1}_{s_0,s_1} = \left(\begin{array}{cc} s_0s_1+1 & s_0 \\ s_1 & 1 \end{array}\right) = \left(\begin{array}{cc} 1 & s_0 \\ 0 & 1 \end{array}\right)\left(\begin{array}{cc} 1 & 0 \\ s_1 & 1  \end{array}\right).
  \]  
  Now, using Proposition \ref{pro:equilatt} in the step from \ref{eq:beforeequilatt} to \ref{eq:afterequilatt} below, we can write
  \begin{align} 
      U^{-1}_{s_0,s_1}A&=\lefteqn{\left(\begin{array}{cc} 1 & s_0 \\ 0 & 1 \end{array}\right)\left(\begin{array}{cc} 1 & 0 \\ s_1 & 1 \end{array}\right)\left(\begin{array}{cc} a & 0 \\ s & b \end{array}\right)}\nonumber \\
      & = \left(\begin{array}{cc} 1 & s_0 \\ 0 & 1 \end{array}\right)\left(\begin{array}{cc} a & 0 \\ s_1a+s & b \end{array}\right)\label{eq:beforeequilatt}\\
      & = \left(\begin{array}{cc} 1 & s_0 \\ 0 & 1 \end{array}\right)\left(\begin{array}{cc} \frac{ab}{X} & ak_1 \\ 0 & X \end{array}\right)\left(\begin{array}{cc} k_2 & -k_1 \\ Y & b/X \end{array}\right) \label{eq:afterequilatt} \\
      & = \left(\begin{array}{cc} \frac{ab}{X} & s_0 X+ak_1 \\ 0 & X \end{array}\right)\left(\begin{array}{cc} k_2 & -k_1 \\ Y & b/X \end{array}\right). \label{eq:lattreform}
  \end{align}
  Here $X=\gcd(s_1 a+s,b)$, $Y = X^{-1}(s_1a+s)$ and $k_1,k_2$ stem from B\'ezout's identity when representing $\gcd(s_1a+s,b)=k_1(s_1a+s)+k_2 b $. It is important to note that the second matrix in the last line has determinant one. This shows that the lattice $U_{s_0,s_1}A$ is separable if and only if $\tilde{D} = \left(\begin{smallmatrix} ab/X & s_0 X+ak_1 \\ 0 & X \end{smallmatrix}\right)$ is equivalent to a diagonal matrix, i.e.
  \begin{equation} \label{eq:modcond}
    \mod(s_0 X + ak_1, ab/X)=0.
  \end{equation}  
 
We will now deduce numbers $s_0$ and $s_1$ satisfying the our needs from the prime factor decomposition of the involved quantities. Therefore we represent $L = \prod_{j=1}^J p_j^{n_j}$ for a fixed set of prime numbers. Since $a$ and $b$ are divisors of $L$ we find their prime factor decompositions to have exponents $\{\alpha_j \}_{j=1}^J$ and $\{\beta_j \}_{j=1}^J$, where $\alpha_j, \beta_j \leq n_j$. The shearing parameter has the decomposition $s = l \prod_{j=1}^J p_j^{\sigma_j}$, where $\gcd(l,L) = 1$.
 
 We choose
 \begin{equation}\label{eq:s1choice}
  s_1 = \prod_{j=1}^J p_j^{\mu_j} \text{, where } \mu_j =
  \begin{cases}
   1 \text{ for } \alpha_j = \sigma_j \\
   0 \text{ else.}
  \end{cases}
 \end{equation} 
 With this choice of $s_1$ we investigate 
 \begin{equation*}
  X = \gcd(s_1a+s,b) = \prod_{j=1}^J \gcd(s_1 a+s,p_j^{\beta_j}).
 \end{equation*}
 To do so, we have to individually treat three cases:
 \begin{enumerate}
  \item $\alpha_j < \sigma_j$: Since $s_1$ and $p_j$ are coprime we find $\gcd(s_1a+s,p_j^{\beta_j})=p_j^{\min(\alpha_j, \beta_j)}$.
  \item $\alpha_j > \sigma_j$: $\gcd(s_1a+s,p_j^{\beta_j})=p_j^{\min(\sigma_j, \beta_j)}$, and $\min(\sigma_j, \beta_j) < \alpha_j$
  \item $\alpha_j = \sigma_j$: Use Eq. \eqref{eq:s1choice} to determine that $\gcd(s_1a+s,p_j^{\beta_j})=p_j^{\min(\alpha_j, \beta_j)}$ %(!!!Do we need more details here? Every other way of writing it would probably take more space!!!)
 \end{enumerate}
 The above arguments show that with the choice of $s_1$, we find that $X = \prod_{j=1}^J p_j^{\gamma_j}$, where $\gamma_j \leq \alpha_j$.

 Now we turn to the choice of $s_0$. To do so we first decompose $k_1 = l\prod_{j=1}^J p_j^{\kappa_j}$, where $l$ and $L$ are coprime. Let us explain how to choose the shear via the positive part of a vector
 \begin{equation*}
  s_0 = \left( \prod_{j=1}^J p_j^{(\beta_j-\gamma_j - \kappa_j)_+} -l \right) \prod_{j=1}^J p_j^{\alpha_j + \kappa_j - \gamma_j}, 
 \end{equation*} 
 where $(x_+)_j = \max(x_j,0)$. A straightforward calculation shows then that
 \begin{equation*}
  s_0 X +a k_1 = \prod_{j=1}^J p_j^{(\beta_j-\gamma_j - \kappa_j)_+ +\alpha_j + \kappa_j},
 \end{equation*} 
 and we furthermore see that
 \begin{equation*}
  (\beta_j-\gamma_j - \kappa_j)_+ +\alpha_j + \kappa_j \geq \beta_j + \alpha_j -\gamma_j.
 \end{equation*} 
 This proofs that \eqref{eq:modcond} is satisfied, completing the proof.
\end{proof}

\begin{remark}
  It is easy to see that $X$ in the proof above satisfies $\gcd(a,b) = kX$ for some $k\in\NN_0$ and therefore $ab/X$ is a multiple of $X$. Thus, the diagonal matrix constructed above is in fact the Smith normal form of $A$.
\end{remark}

\subsection{Further optimization}

In this section we will first determine which signal lengths are feasible for some given choice of $a,M$ and $\lambda_1, \lambda_2$. This restriction holds for all the presented methods equally and is essential to know in computations.

Particularly when using the shear method described in Subsection \ref{ssec:shears} it is interesting to know for which signal lengths one of the two shears $s_0$ and $s_1$, preferably the frequency side shear $s_0$, can be chosen to be zero. This saves additional computation time.

\begin{proposition} \label{prop:minsig}
 Given the parameters $\lambda=\lambda_1/ \lambda_2$, $a$ and $M$. Then the minimal signal length, for which these parameters are feasible is given by $L_{\text{min}}=\lambda_2\lcm{a,M}$. All the feasible signal lengths are multiples of this.
\end{proposition}
\begin{proof}
 For the parameters in combination with a given signal length $L$ to form a lattice we require the following conditions
 \begin{equation*}
  \begin{split}
   a|L, \; M|L \\
   \frac{L}{a} \lambda \in \ZZ \\
   \frac{L}{M} \lambda \in \ZZ,
  \end{split}
 \end{equation*}
where the first conditions immediately yield $\lcm{a,M}|L$. From the other two conditions we can derive 
\[ 
  a\lambda_2/\gcd(a,\lambda_1) | L\quad \text{ and }\quad M\lambda_2/\gcd(M,\lambda_1) | L. 
\]
Therefore, the signal length has to be a multiple of
\begin{equation*}
 L_{\text{min}} = \lcm{\frac{a\lambda_2}{\gcd(a,\lambda_1)},\frac{M\lambda_2}{\gcd(M,\lambda_1)},a,M}.
\end{equation*} 
We proceed to show that 
\begin{equation} \label{eq:gcda}
\lcm{\frac{a\lambda_2}{\gcd(a,\lambda_1)},a} = \lambda_2a.  
\end{equation} 
For this purpose we look at the prime factor decomposition of the involved quantities, where we denote by $\alpha_j, \gamma_j, \delta_j$ the exponents of the prime number $p_j$ of $a, \lambda_1$ and $\lambda_2$ respectively. Then we find, since $\lambda_1$ and $\lambda_2$ are co-prime that the exponent of $p_j$ of $\lcm{a\lambda_2/\gcd(a,\lambda_1),a}$ is given by
\begin{equation*}
 \max(\alpha_j - \min(\alpha_j,\gamma_j)+\delta_j, \alpha_j) = \alpha_j+\delta_j,
\end{equation*} 
proving \eqref{eq:gcda}. The proof that $\lcm{M\lambda_2/\gcd(M,\lambda_1),M} = \lambda_2 M$ is completely analogous. Combine these to find
\begin{equation*}
 \begin{split}
  \lefteqn{\lcm{\frac{a\lambda_2}{\gcd(a,\lambda_1)},\frac{M\lambda_2}{\gcd(M,\lambda_1)},a,M}}\\
      &= \lcm{\lambda_2a,\lambda_2M} = \lambda_2 \lcm{a,M} .
 \end{split}
\end{equation*} 
\end{proof}

Now we shall investigate, which multiples of the just derived minimal signal length allow for computation without the frequency shear. To do so, it is instructive to compute the set of factors $l$, for which $L = l L_{\text{min}}$ needs only the time shear. 
We will introduce here some important constants
related to the {\em time shift} $a$, the {\em frequency shift} $b$, the {\em number of channels} $M=L/b$ and the {\em number of time shifts} $N=L/a$. We define $c,d,p,q\in\mathbb{N}$ by
\begin{eqnarray}
c=\gcd\left(a,M\right) & , & d=\gcd\left(b,N\right),\label{eq:cd}\\
p=\frac{a}{c}=\frac{b}{d} & , & q=\frac{M}{c}=\frac{N}{d}.\label{eq:pq}
\end{eqnarray}
With these numbers, the \emph{redundancy} of a Gabor system can be
written as $L/\left(ab\right)=q/p,$ where $q/p$ is an irreducible
fraction. It holds that $L=cdpq$. Some of the introduced notation will be important in the next section.

\begin{proposition}\label{pro:noshearlength}
  Given $\lambda$, $a$ and $M$. Let the prime factor decomposition of $c= \gcd(a,M)$ be given by
  \begin{equation*}
    c = \prod_{j=1}^J p_j^{\gamma_j}
  \end{equation*}
  for some set of prime factors and corresponding exponents. Let
  \begin{equation*}
    c_1= \prod_{j=1}^J p_j^{\sigma_j}, \quad
    \sigma_j=
    \begin{cases}
      \gamma_j \quad & \text{if $\gcd(\lambda_2, p_j)=0$} \\
      0 & \text{else},
    \end{cases}
  \end{equation*}
  then the frequency shear can be chosen to be $0$ if the signal length satisfies
  \begin{equation}\label{eq:noshearlength}
    L = n L_\text{min} \frac{c}{c_1},
  \end{equation}
  for some $n\in \NN$. In words, $c_1$ are factors of $c$ that are relatively prime to $\lambda_2$
\end{proposition}

\begin{proof}
 With the standard notation we easily see that the time shear is sufficient if and only if $(s+kb)/a \in \ZZ$ for some $k\in \{0, \ldots, M-1\}$. Rewriting this leads to
 \begin{equation*}
  L = \tilde l \frac{Ma\lambda_2}{\lambda_1+k \lambda_2} = l\frac{Ma\lambda_2}{\gcd(\lambda_1+k \lambda_2,Ma \lambda_2)},
 \end{equation*} 
 for some $l\in \ZZ$. However, these signal lengths might not be compatible with the feasibility condition from Proposition \ref{prop:minsig}. Therefore we compute the ratio
 \begin{equation*}
  \frac{L}{L_{\text{min}}} = l \frac{\gcd(M,a)}{\gcd(\lambda_1+k \lambda_2,Ma \lambda_2)}.
 \end{equation*} 
 Since this fraction should be an integer number, we have to choose
 \begin{equation*}
   l = n \frac{\gcd(\lambda_1 + k \lambda_2, Ma \lambda_2)}{\gcd(M,a, \lambda_1 + k\lambda_2, M a \lambda_2)},
 \end{equation*}
 for some $n\in \NN$. Therefore, we can compute
 \begin{equation}\label{eq:fullset}
   L= n L_\text{min} \frac{\gcd(M,a)}{\gcd(M,a,\lambda_1 + k \lambda_2)}.
 \end{equation}

 With the notation introduced above we are now interested in computing
 \begin{equation}\label{eq:restrset}
   \max_{k\in \NN} (\gcd(c,\lambda_1 + k \lambda_2)).
 \end{equation}
 Firstly, we rewrite
 \begin{equation*}
   \gcd(c,\lambda_1 + k \lambda_2) = \prod_{j=1}^J \gcd(p_j^{\gamma_j},\lambda_1 + k\lambda_2).
 \end{equation*}
 Now we will individually investigate the factors in the product above.

 \textbf{Case 1.} $\gcd(p_j,\lambda_2)=1$, in which case we can find numbers $k_{1,j}, k_{2,j}$, such that
 \begin{equation*}
   \lambda_1 + k_{1,j} \lambda_2 = k_{2,j} p_j^{\gamma_j}.
 \end{equation*}
 Furthermore, the full set of coefficients of $\lambda_2$, for which the above equation can be satisfied is given by $K_j = \left\{ k_{1,j} + m p_j^{\gamma_j}: m\in \ZZ  \right\}$.
 Therefore, for any $k \in K_j$ we find
 \begin{equation*}
   \gcd(p_j^{\gamma_j},\lambda_1 + k\lambda_2)= p_j^{\gamma_j}.
 \end{equation*}

 \textbf{Case 2.} $\gcd(p_j, \lambda_2)\neq 1$, which implies directly that $\lambda_2$ is a multiple of $p_j$.
 In this case we have to argue that $\lambda_1 + k \lambda_2$ can never be multiple of $p_j$. 
 Indeed, any linear combination $k_{1,j}\lambda_2+ k_{2,j} p_j$ is a multiple of $p_j$ and therefore not equal to $\lambda_1$, which is assumed to be relatively prime to $\lambda_2$.
 Consequently, for any choice of $k\in \ZZ$
 \begin{equation*}
   \gcd(p_j^{\gamma_j},\lambda_1 + k\lambda_2)= 1.
 \end{equation*}

 For all the indices $j$ in case $1$, it is easy to see that the intersection of the corresponding sets $K_j$ is not empty.
 This is an immediate consequence from the fact that powers of two different prime numbers have no common divisors.
 Using the notation introduced above, we can conclude that there exists some $k \in \ZZ$, such that
 \begin{equation*}
   \gcd(c, \lambda_1+\lambda_2) = c_1.
 \end{equation*}
 The last argument needed is to show that $k\in \left\{ 0, \ldots, M-1 \right\}$.
 By construction $s+kb = \tilde k a$, for some $\tilde k \in \ZZ$.
 Therefore, for any $m\in \ZZ$
 \begin{equation*}
   s+\left( k+ m \frac{L}{b} \right)b = \left( \tilde k + m \frac{L}{a} \right) a,
 \end{equation*}
 and for an appropriate choice of $m$, the expression in brackets on the left hand side will evaluate some number in the desired range.

\end{proof}

\begin{remark}
  There are possibly other feasible signal lengths than determined by
  \eqref{eq:noshearlength}. The full set of feasible lengths is determined by
  \begin{equation*}
    \left\{ n L_\text{min} \frac{\gcd(M,a)}{\gcd(M,a,\lambda_1 + k \lambda_2)}: \; n\in \NN, k\in \left\{ 0, \ldots, M-1 \right\} \right\}.
  \end{equation*}
  This can be easily seen from \eqref{eq:fullset} in the proof above. 
  For simplicity we only construct the minimal factor, that $L_\text{min}$
  has to be multiplied with, as stated in \eqref{eq:restrset}.
\end{remark}
\begin{remark}
  Looking at \eqref{eq:noshearlength} we see that if we are given a
  certain redundancy \eqref{eq:pq} $q/p$ and a lattice type
  $\lambda_1/\lambda_2$ and want to get a low value of $L_{min}c/d$ we
  must choose $c$ such that it is relatively prime to $\lambda_2$. As an
  example, consider a common choice of $a=32$, $M=64$ and
  $\lambda_1/\lambda_2=1/2$ (the quincunx lattice). In this case
  $c=\gcd(a,M)=32$ which is the worst possible case, as it is a power of
  $\lambda_2=2$ giving a value of $L_{min}c/d=128\cdot 32=4096$. If we
  instead choose $a=27$, $M=54$ (which is the same redundancy) we get
  $L_{min}c/d=108\cdot 1=108$. This illustrates that it is possible to
  work efficiently with the quincunx lattice by not choosing the
  rectangular lattice parameters to be powers of 2.
\end{remark}

\subsection{Extension to higher dimensions}

  It is well known~\citep{pa07,chfepa11,feka97} that multidimensional Gabor transforms
  and dual windows can be computed using algorithms designed for the $1$D case, if both the
  Gabor window and the lattice used can be written as a tensor product. That is, we assume that 
  with $l = (l_1,\ldots,l_n)^T \in \CC^{L_1}\times\ldots\times \CC^{L_n}$,
  \[
   g(l) = g_1(l_1) \otimes \ldots \otimes g_n(l_n)
  \]
  and
  \[
   \Lambda = \Lambda_1 \times \ldots \times \Lambda_n = A_1\ZZ^2_{L_1} \times \ldots \times A_n\ZZ^2_{L_n}
  \]
  for some $A_j\in\ZZ^2_{L_j}\times\ZZ^2_{L_j}$ for $j = 1,\ldots,n$.
  
  Equivalently, we can say that $\Lambda$ can be described by a block matrix 
  \begin{equation}\label{eq:biglattmatrix}
   A = \left(\begin{array}{cc}
          D & E \\ F & G
         \end{array}\right).
  \end{equation}
  with diagonal blocks $D,E,F,G\in\ZZ^{n\times n}$.
  In this case, the multidimensional transform and dual window can be computed by subsequently applying 
  the algorithms presented in the previous sections in every dimension. A matrix describing the lower dimensional
  lattice corresponding to dimension $j$ is simply given by
  \[
  A_j = \left(\begin{array}{cc}
          D_{j,j} & E_{j,j} \\
          F_{j,j} & G_{j,j}         
         \end{array}\right)
  \]
  and can be transformed into lattice normal form \eqref{eq:lattNF}, allowing straightforward application 
  of the presented algorithms. 
  
  However, we are not aware of a constructive method to determine whether a lattice, given by an arbitrary matrix, 
  can be described by a banded matrix of the form \eqref{eq:biglattmatrix}.

  In contrast to the methods based on metaplectic operators, it is easier to extend the multiwindow approach from \ref{ssec:multiwin} to higher dimensions. For reasons of readability we will not give details here.

\section{Implementation and timing}\label{sec:Implementation}

In this section we discuss the implementation and speed of the proposed algorithms.
It is important for this section to recall the definition of the constants $c,d$
in \ref{eq:cd} and $p,q$ in \eqref{eq:pq}.\\

\noindent\textbf{Methodology for computing the computational complexity.}
To compute the Discrete Fourier transform, the familiar FFT algorithm
is used. When computing the flop (floating point operations) count of
the algorithm, we will assume that a complex FFT of length $M$ can be
computed using $4M\log_{2}M$ flops. A review of flop counts for FFT
algorithms is presented in \citep{jofr07}. When computing the flop
count, we assume that both the window and signal are complex valued.

The cost of performing the computation of a DGT with a full length
window on a rectangular lattice using the algorithm first reported in
\citep{ltfatnote011} is given by
\begin{align}
 \lefteqn{8Lq+4L\log_{2}d+4MN\log_{2}d+4MN\log_{2}\left(M\right)}\label{eq:rect-flops}\\
 & = L\left(8q+4\log_{2}d\right)+4MN\left(\log_{2}L/p\right)\hspace{40pt}\label{eq:rect-flops-cleaned}
\end{align}
where the first terms in \eqref{eq:rect-flops} come from the
multiplication of the matrices in the factorization, the two middle
terms come from creating the factorization of the signal and
inverting the factorization of the coefficients, and the last term
comes from the final application of FFTs. The terms can be collected
as in \eqref{eq:rect-flops-cleaned}, where the first term grows as the
length of the signal $L$, and the second terms grows as the total
number of coefficients $MN$. In the following, we refer to this as the {\em full window} algorithm.

If the window is an FIR window supported on an index set with width
$L_g$ which is much smaller than the length of the signal $L$, the
{\em weighted-overlap-add algorithm}, first reported in \citep{po76}, can
be used instead. It has a computational complexity of
\begin{equation*}
8L\frac{L_{g}}{a}+4NM\log_{2}M.\label{eq:rect-flops-FIR}
\end{equation*}
In the following, we refer to this as the {\em FIR window} algorithm.

A third approach to computing a DGT is a hybrid approach, where a DGT
using an FIR window can be computing using a full window algorithm on
blocks of the input signal. The blocks are then combined using the
classical {\em overlap-add} algorithm, cf. \citep{stockham1966high,helms1967fast}.

The OLA algorithm works by partitioning a system of length $L$ into
blocks of length $L_{b}$ such that $L=L_{b}N_{b}$, where $N_{b}$
is the number of blocks. The block length must be longer than the
support of the window, $L_{b}>L_{g}$. To perform the computation
we take a block of the input signal of length $L_{b}$ and zero-extend
it to length $L_{x}=L_{b}+L_{g}$, and compute the convolution with the
extended signal using the similarly extended window. Because of the
zero-extension of the window and signal, the computed coefficients
will not be affected by the periodic boundary conditions, and it is
therefore possible to overlay and add the computed convolutions of
length $L_{x}$ together to form the complete convolution of length
$L$.

The equations \eqref{eq:rect-flops-cleaned}, \eqref{eq:rect-flops-FIR}
are used to express the efficiency of the algorithms for the DGT on
nonseparable lattices.

\subsection{Implementation of the shear algorithm}

The shear algorithm proposed in Proposition \ref{pro:sheartrans} 
computes the DGT on a nonseparable lattice
using a DGT on a separable lattice with some suitable pre- and
postprocessing steps. The computational complexity of the pre- and
postprocessing steps is significant compared to the separable DGT,
so we wish to minimize the cost of these steps. An implementation of the shear 
algorithm is presented as Algorithm \ref{alg:shear}. Note that we assume the existence 
of several underlying routines: An implementation \textsc{dgt} of the separable Gabor transform, 
the periodic chirp \textsc{pchirp}(L,s)$ = \exp(\pi i s\cdot^2 (L+1)/L)$ and \textsc{shearfind}, a
program that determines the shear parameters $s_0,s_1$ and the correct separable lattice to do the 
DGT on, following the constructive proof of Theorem \ref{thm:shearform}.

\begin{algorithm}[t!ph]
\caption{The shear algorithm: $c = \textsc{dgtns}(f,g,a,M,\lambda)$}
\label{alg:shear}
%\begin{small}
\begin{algorithmic}[1]
\State $\left[s_0,s_1,b_r\right]=\textsc{shearfind}(L,a,M,\lambda)$
\If{$s_1 \neq 0$}
    \State $p \gets \textsc{pchirp}(L,s_1)$
    \State $g(\cdot) \gets  p(\cdot)g(\cdot)$
    \State $f(\cdot) \gets  p(\cdot)f(\cdot)$
\EndIf

\If{$s_0 = 0$}

    \State $c_r \gets  \textsc{dgt}(f,g,a,M)$
    \State $C_1\gets s_1 a (L+1) \pmod{2N}$

    \For{$k=0\to N-1$}   
        \State $E \gets  e^{\pi i (C_1 k^2 \pmod{2N})/N}$
        \For{$m=0\to M-1$}
            \State $c(\left\lfloor\frac{-s_1 k a+m b \pmod{L}}{b}\right\rfloor,k) \gets  Ec_r(m,k)$
        \EndFor
    \EndFor
    
\Else 
    \State $a_r\gets \frac{ab}{b_r}, \quad M_r\gets \frac{L}{b_r}, \quad  N_r\gets \frac{L}{a_r}$
    \State $C_1\gets \frac{a_r}{a},\quad C_2\gets -s_0 b_r/a$
    \State $C_3\gets as_1(L+1),\quad C_4\gets C_2 b_r(L+1)$
    \State $C_5\gets 2 C_1 b_r,\quad C_6\gets (s_0 s_1+1) b_r$
    
    \State $p \gets  \textsc{pchirp}(L,-s_0)$
    \State $g(\cdot) \gets  p(\cdot)\textsc{fft}(g(\cdot))/L$
    \State $f(\cdot) \gets  p(\cdot)\textsc{fft}(f(\cdot))$
    
    \State $c_r \gets  \textsc{dgt}(f,g,b_r,N_r)$
    
    \For{$k=0\to N_r-1$}
        \For{$m=0\to M_r-1$}
            \State $s_{q1}\gets C_1k+C_2 m \pmod{2N}$
            \State $E \gets  e^{\pi i (C_3 s_{q1}^2-m (C_4 m+C_5k) \pmod{2N})/N}$
            
            \State $\tilde m \gets  C_1k       +C_2 m \pmod{N}$
            \State $\tilde k \gets  \left\lfloor \frac{-s_1 a_r k +C_6 m \pmod{L}}{b}\right\rfloor$
            
            \State $c(\tilde k,\tilde m) \gets  Ec_r(-k \pmod{N_r},m)$
        \EndFor
    \EndFor                    
\EndIf
\end{algorithmic}
%\end{small}
\end{algorithm}

A simple trick is to notice that when a frequency-side shear is
needed, the DFT of the signal $f$ and the window $g$ are multiplied by
a chirp on the frequency side, $\tilde{f} =
\bd{U}_{s_0,s_1}^{-1}f$ and $\tilde{g} =
\bd{U}_{s_0,s_1}^{-1}g$. The total cost of this is 4 FFT's and two
pointwise multiplications. However, instead
of transforming the chirped signal and window back to the time domain,
we can compute the nonseparable DGT directly in the frequency domain
using the well-known commutation relation of the DFT and the
translation and modulation operators:
\begin{eqnarray}
\left\langle f,M_{m}T_{n}g\right\rangle &=&e^{-\pi imn/L}\left\langle \mathcal{F}f,M_{-n}T_{m}\mathcal{F}g\right\rangle 
\end{eqnarray}
This trick saves the two inverse FFTs at the expense of the
multiplication of the coefficients by a complex exponential and
reshuffling. As we already need these operations to realize 
\eqref{eq:shearcoeffs} and \eqref{eq:shearcoeffs-2}, they can be
combined with no additional computational complexity.

The overlap-add algorithm can be used in conjunction with the shear
algorithm in the following case: we wish to compute the DGT with an
FIR window for a nonseparable lattice using the shear
algorithm. Because of the frequency-side shearing, the window is
converted from an FIR window into a full length window, making it
impossible to perform real-time or block-wise processing. However, if
the shear algorithm is used inside an OLA algorithm, this is no longer
a concern, as the shearing will only convert the window into a window
of length $L_{g}+L_{b}$, restoring the ability to perform block-wise
processing.

In total, the shear-OLA algorithm for the DGT is calculated in three steps using the three algorithms:

\begin{enumerate}
\item Split the input signal into blocks using the overlap-add algorithm
\item Apply the shears to the blocks of the input signal as in the shear algorithm
\item Use the full-window rectangular lattice DGT on the sheared signal blocks.
\end{enumerate}

The downside of the shear-OLA algorithm is that the total length of
the DGTs is longer than the original DGT by
\begin{equation}
\rho=\frac{L_{g}+L_{b}}{L_{b}},\label{eq:OLA-eff} %=\frac{L_{x}}{L_{b}} L_x is never used!
\end{equation}
where $L_{b}$ is the block length. Therefore, a trade-off between the
block length and the window length must be found, so that the block
length is long enough for \eqref{eq:OLA-eff} to be close to one, but
at the same time small enough to not impose a too long processing
delay.

\subsection{Dual and tight windows}

The shear method in Proposition \ref{pro:sheartrans} can also be
used to compute the canonical dual and canonical tight windows, using
the factorization of the frame operator given in 
\eqref{eq:shearframeop}. The complete algorithm for the canonical dual
window is shown in \ref{alg:sheardual}, and uses the same trick as
the Gabor transform algorithm to compute the canonical dual when a
frequency side shear is needed: do it in the Fourier domain without
transforming back. Again, we assume the existence of an implementation
\textsc{gabdual} for the computation of Gabor dual windows on separable lattices.

\begin{algorithm}[t!]
\caption{Dual window via shearing: \\ $\tilde{g} = \textsc{gabdualns}(g,a,M,\lambda)$}
\label{alg:sheardual}
%\begin{small}
\begin{algorithmic}[1]
\State $\left[s_0,s_1,b_r\right]=\textsc{shearfind}(L,a,M,\lambda)$

\If{$s_1 \neq 0$}
    \State $p \gets \textsc{pchirp}(L,s_1)$
    \State $g(\cdot) \gets  p(\cdot)g(\cdot)$
\EndIf

\State $b   \gets \frac{L}{M},\quad M_r \gets \frac{L}{b_r},\quad a_r \gets \frac{ab}{b_r}$

\If{$s_0 = 0$}

\State $g_d\gets\textsc{gabdual}(g,a_r,M_r)$
\Else 

\State $p_0 \gets\textsc{pchirp}(L,-s_0)$
\State $g(\cdot) \gets p_0(\cdot)\textsc{fft}(g)(\cdot)$
\State $g_d \gets L\cdot\textsc{gabdual}(g,L/M_r,L/a_r)$
\State $g_d \gets \textsc{ifft}(\overline{p_0}(\cdot)g_d(\cdot))$       

\EndIf
\end{algorithmic}
%\end{small}
\vspace{-0pt}
\end{algorithm}

To compute the canonical dual and tight windows on a separable
lattice, the matrices are first factorized as in
\citep{st98-8,ltfatnote011} and then the factorized matrices are
transformed as in \citep{ltfatnote007}.

\subsection{Analysis of the computational complexity}\label{sec:Computational-complexity}

\begin{table}[t!h]
\caption{\label{tab:Flop-counts}Flop counts for different ways of
  computing the DGT on a nonseparable lattice. First column list the
  algorithm, second column the flop count for the particular
  algorithm.  Listed from the top, the algorithms are: The
  multiwindow algorithm using the full window rectangular lattice
  algorithm, the multiwindow algorithm using the FIR window
  rectangular lattice algorithm, the Smith normal form algorithm using
  the full window rectangular lattice algorithm, the shear algorithm
  when no frequency shear is needed, the shear algorithm including the
  frequency shear and finally the overlap-add versions of the shear
  algorithms.
  The term $L_{g}$ denotes the length of the window used so
  $L_{g}/a$ is the overlapping factor of the window.}

\definecolor{Gray}{gray}{0.9}
\hfill{}%
\begin{small}
%\begin{tabular}{|l|>{\centering}p{10cm}|}
\begin{tabular}{|l|c|}
\hline
Alg.: & Flop count\tabularnewline \hline
\hline %\noalign{\vskip\doublerulesep}
\rowcolor{Gray} Multi-window & \tabularnewline

FIR. & $8L\frac{L_{g}}{a}+4NM\log_{2}M$\tabularnewline [.3ex]
%\noalign{\vskip\doublerulesep} 
\rowcolor{Gray}
Full. &
$L\lambda_{2}\left(8q_{mw}+4\log_{2}d_{mw}\right)$\\ \rowcolor{Gray} & $+MN\left(4\log_{2}L/p_{mw}+6\right)$\tabularnewline [.3ex]
\hline %\noalign{\vskip\doublerulesep} 
SNF &
$L\left(8q+4\log_{2}d_{sm}+8\log_{2}L+18\right)$\\ & $+MN\left(4\log_{2}L/p+6\right)$\tabularnewline [.3ex]
\hline %\noalign{\vskip\doublerulesep} 
\rowcolor{Gray} Shear alg. & \tabularnewline
No freq. shear &
$L\left(8q+4\log_{2}d+6k_{time}\right)$\\ & $+MN\left(4\log_{2}L/p+6k_{time}\right)$\tabularnewline [.3ex]
%\noalign{\vskip\doublerulesep} 
\rowcolor{Gray} Freq. shear &
$L\left(8q+4\log_{2}Lc_{sh}+6+6k_{time}\right)$\\ \rowcolor{Gray} & $+MN\left(4\log_{2}L/p+6\right)$\tabularnewline [.3ex]
\hline %\noalign{\vskip\doublerulesep} 
\rowcolor{Gray} Shear OLA & \tabularnewline
No freq. shear&
$\rho L\left(8q+4\log_{2}\rho d_{shola}+6k_{time}\right)$\\ & $+\rho MN\left(4\log_{2}\rho L_b/p+6k_{time}\right)$
  \tabularnewline [.3ex]
%\noalign{\vskip\doublerulesep}
\rowcolor{Gray} Freq. shear&
$\rho L\left(8q+4\log_{2}\rho Lc_{shola}+6k_{time}+6\right)$\\ \rowcolor{Gray} & $+\rho MN\left(4\log_{2}\rho L_b/p+6\right)$
  \tabularnewline [.3ex]
\hline
\end{tabular}\hfill{}
\end{small}
%\vspace{10pt}
%\smallskip{}
\end{table}

The flop counts of the various algorithms used for computing the DGT
on a nonseparable lattice is listed in Table \ref{tab:Flop-counts}.

Based on the computational complexity presented in the table, any of
the algorithms may for some specific problem setup be the fastest,
except for the Smith-normal form algorithm which is always slower than
the shear algorithm:

\begin{itemize}

\item The multiwindow algorithm for FIR windows is the fastest for
  very short windows.

\item The multiwindow-OLA algorithm is the fastest for simple
  lattices ($\lambda_2$ small) and medium length windows.

\item The shear-OLA algorithm is the fastest for more complex
  lattices ($\lambda_2$ large) and medium length windows.

\item The multiwindow algorithm is the fastest for simple lattices
  ($\lambda_2$ small) and very long windows.

\item The shear algorithm is the fastest for more complex
  lattices ($\lambda_2$ large) and very long windows.

\end{itemize}

\subsection{Numerical experiments}\label{sec:Numerical experiments}

Implementations of the algorithms described in this paper can be found
in the Large Time Frequency Analysis Toolbox (LTFAT), cf. \citep{ltfatnote015},\citep{ltfatweb}.
An appropriate algorithm will be automatically invoked when calling
the \textsc{dgt} or \textsc{dgtreal} functions. The implementations
are done in both the \noun{Matlab} / \noun{Octave} scripting language
and in C. All tests were performed on an Intel i7 CPU operating at 3.6 GHz.

As the speed of the algorithms depends on a large number of parameters
$a$, $M$, $L$, $L_g$, $c$, $d$, $s_0$, $s_1$ and similar parameters
relating to the multiwindow and shear transforms, we cannot provide
an exhaustive illustration of the running times. Instead we will
present some figures that illustrates the crossover point of when the
the shear algorithm becomes faster than the multiwindow algorithm as
the lattice complexity $\lambda_2$ increases. The behavior of the
algorithms as the window length $L_g$ increases is completely
determined by the algorithms for the rectangular lattice, so we refer
to \citep{ltfatnote011} for illustrations.

\begin{figure}  
  \includegraphics[height=0.32\textwidth,width=0.45\textwidth,trim=80 30 30 0, clip]{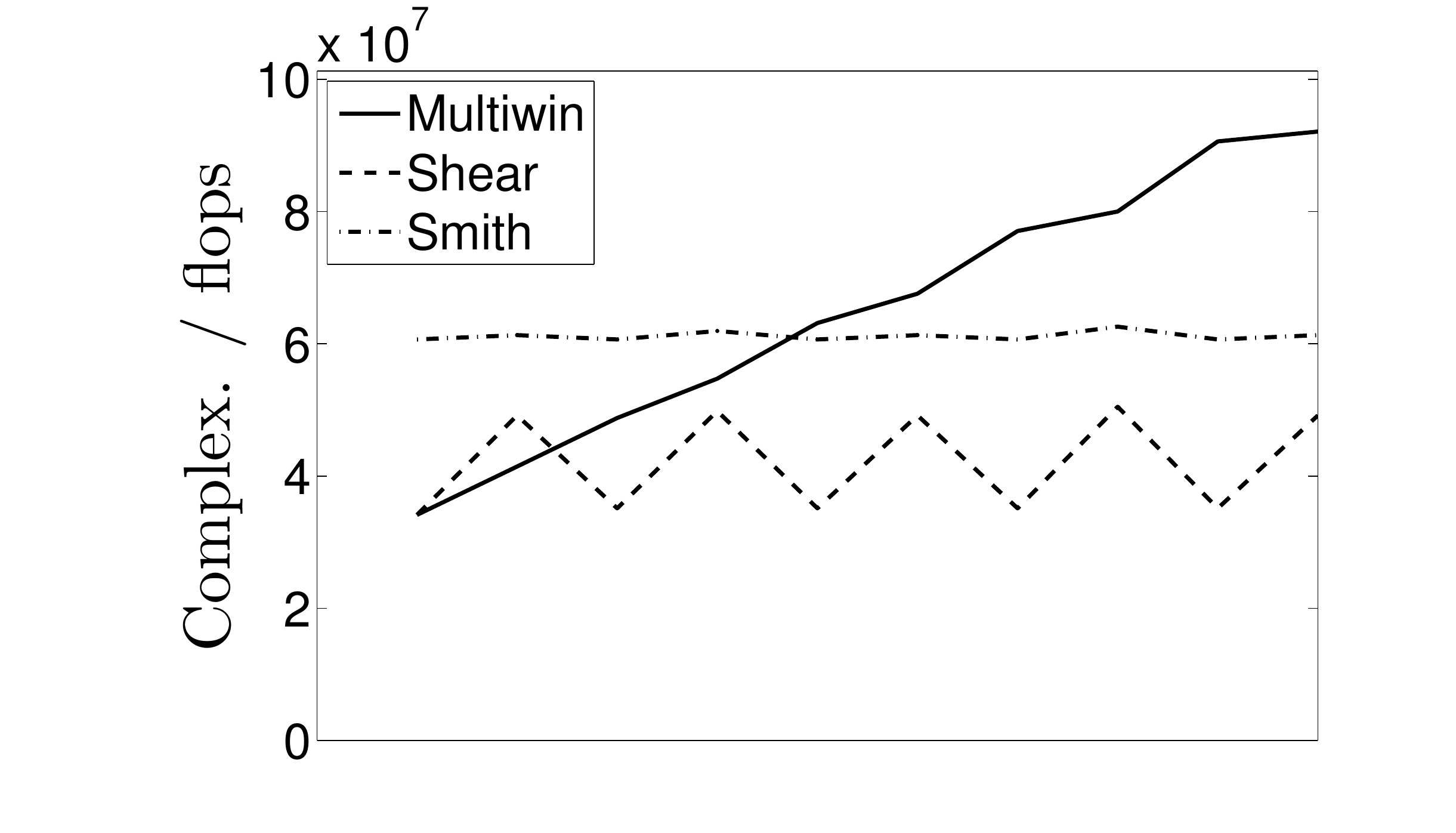}\includegraphics[height=0.32\textwidth,width=0.45\textwidth,trim=100 30 30 0, clip]{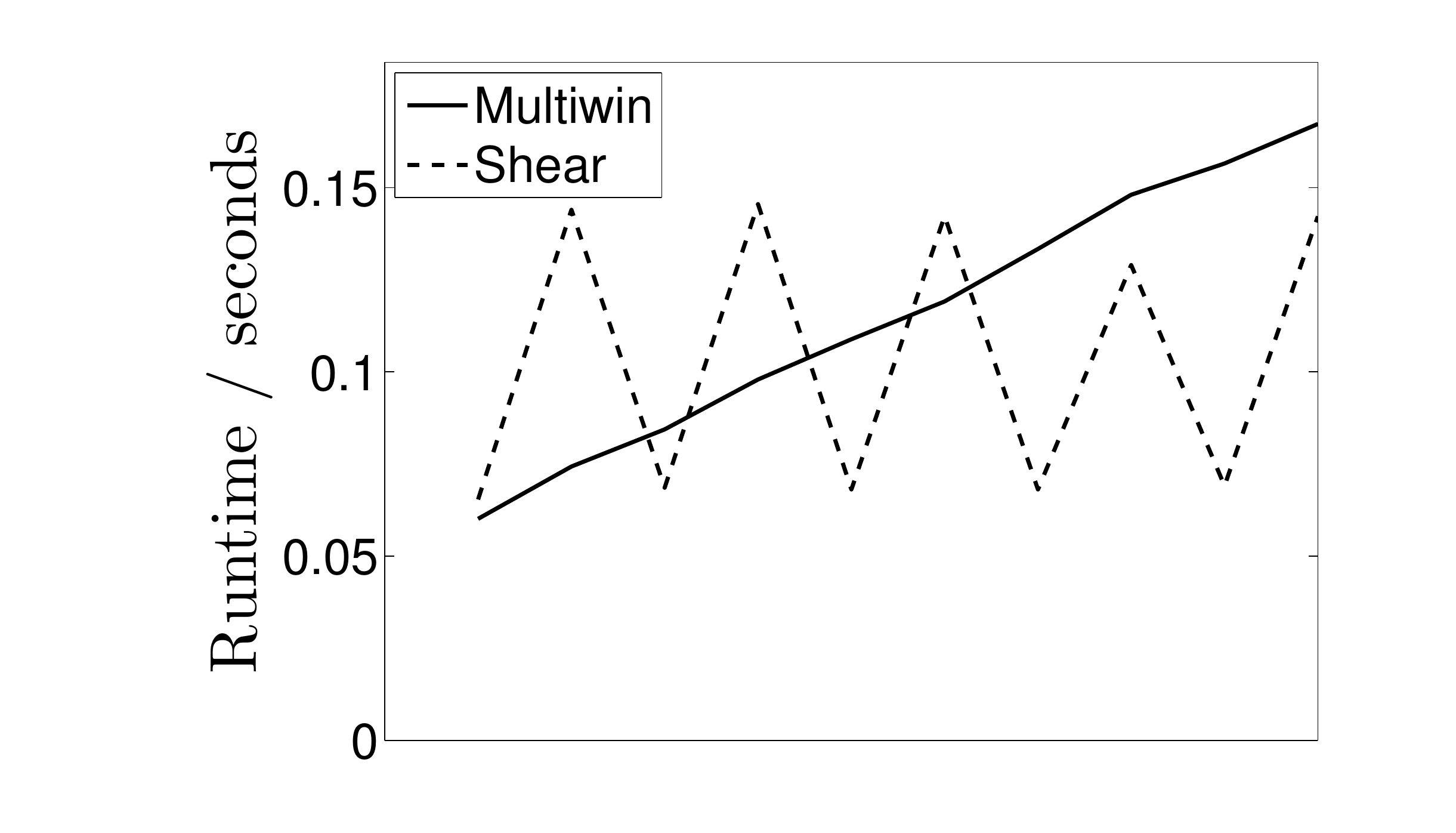}\\  
  \includegraphics[height=0.32\textwidth,width=0.45\textwidth,trim=80 30 30 0, clip]{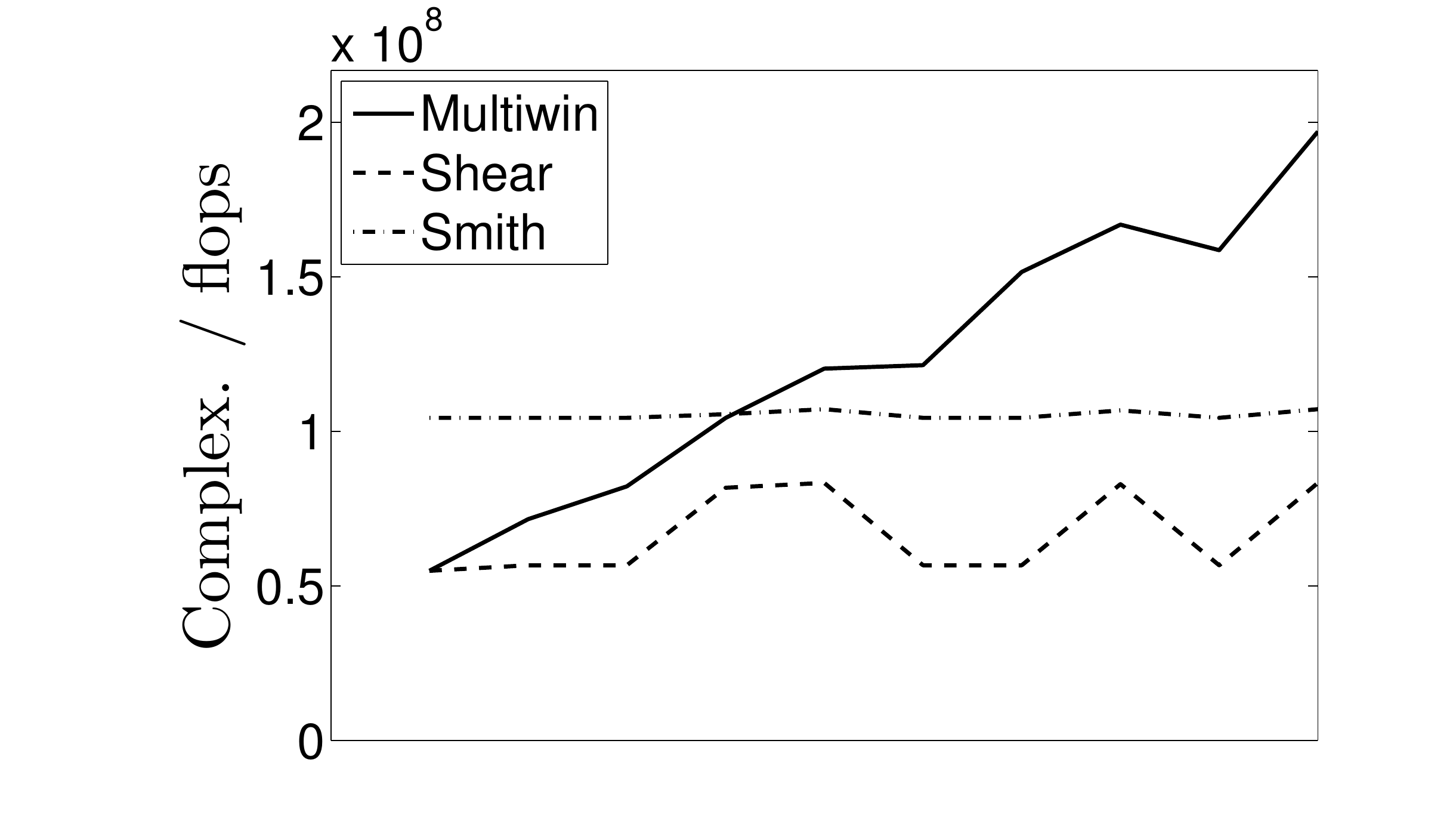}\includegraphics[height=0.32\textwidth,width=0.45\textwidth,trim=100 30 30 0, clip]{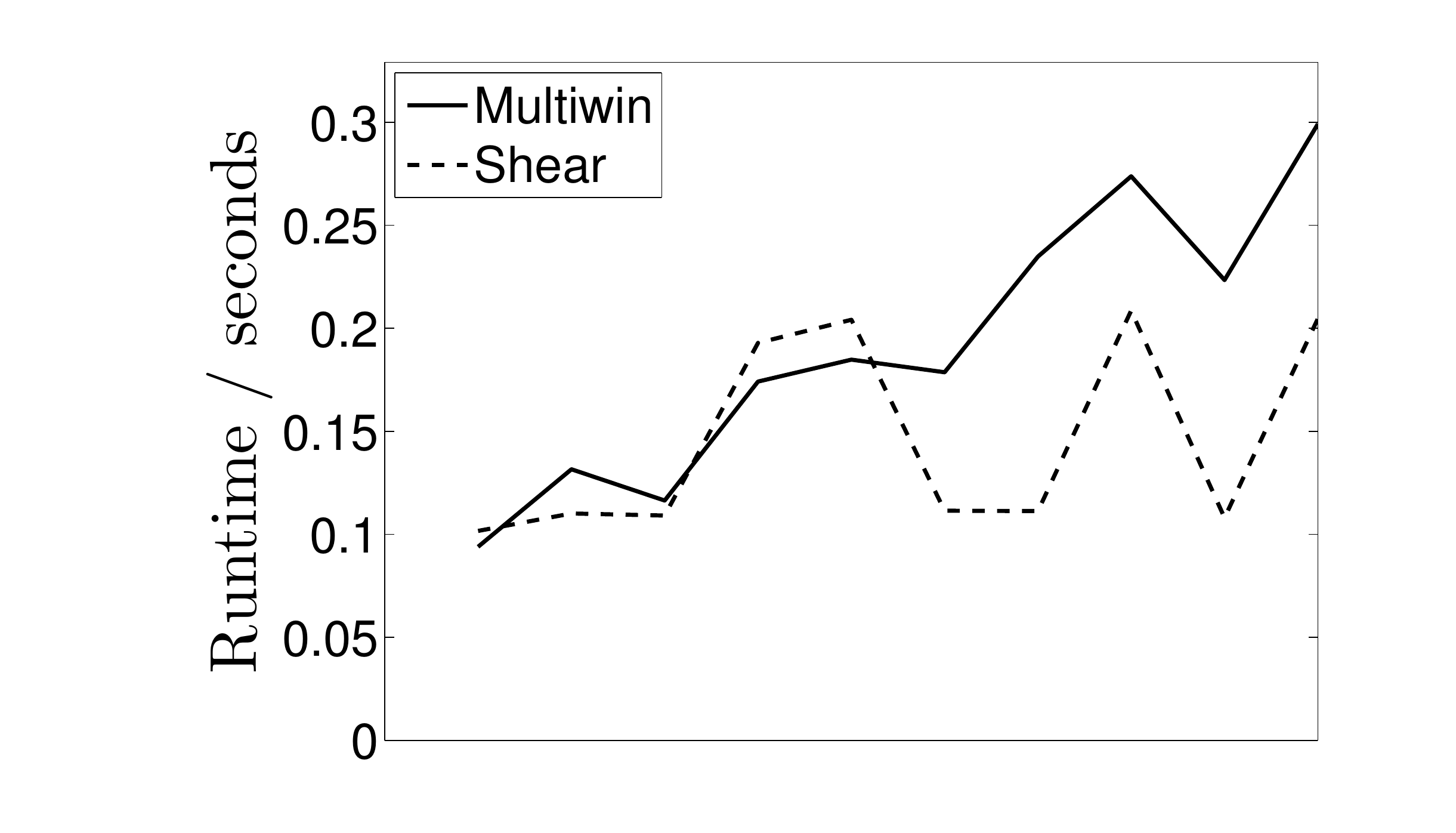}\\
  \includegraphics[height=0.34\textwidth,width=0.45\textwidth,trim=60 0 30 0, clip]{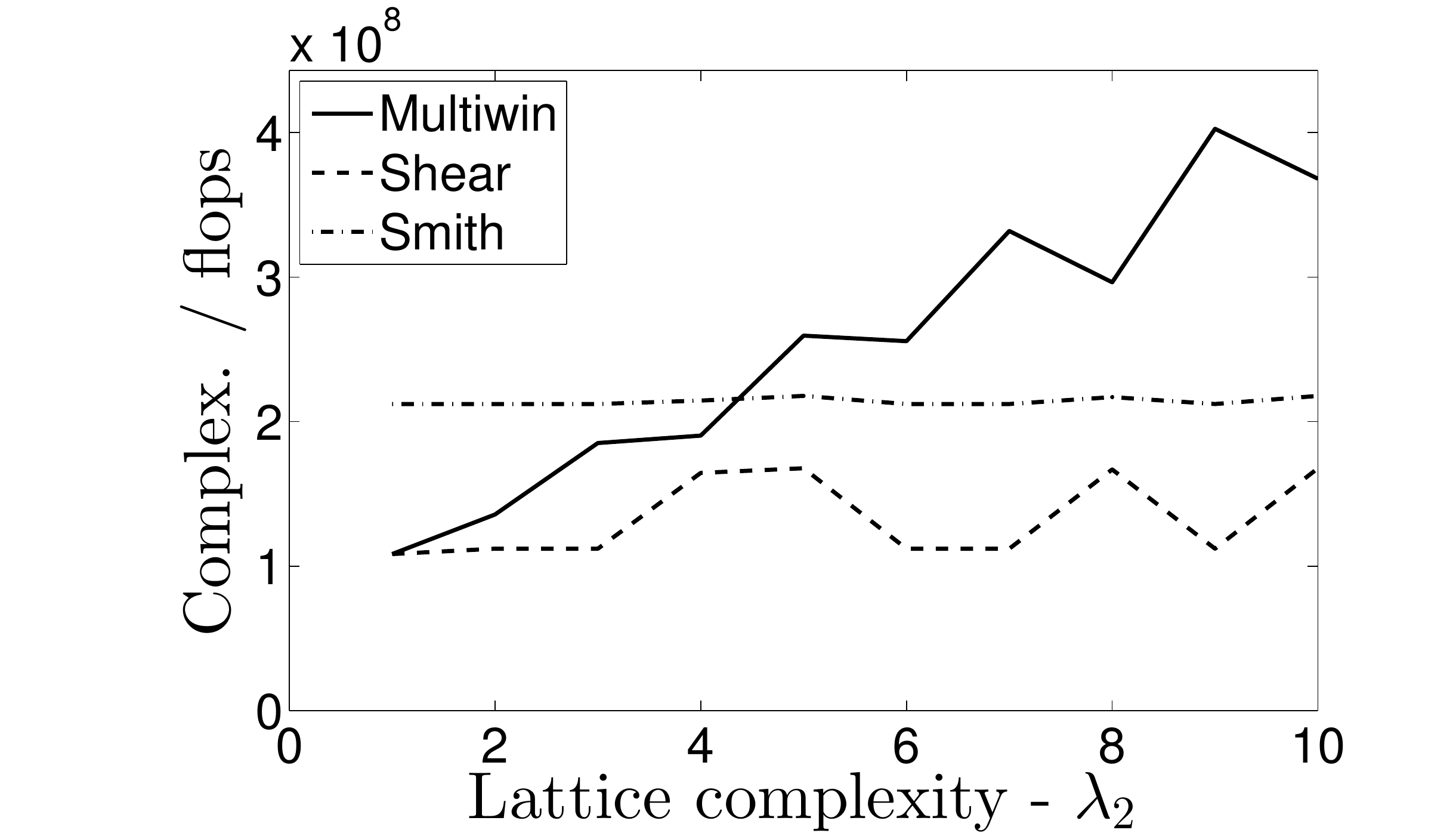}\includegraphics[height=0.34\textwidth,width=0.45\textwidth,trim=80 0 30 0, clip]{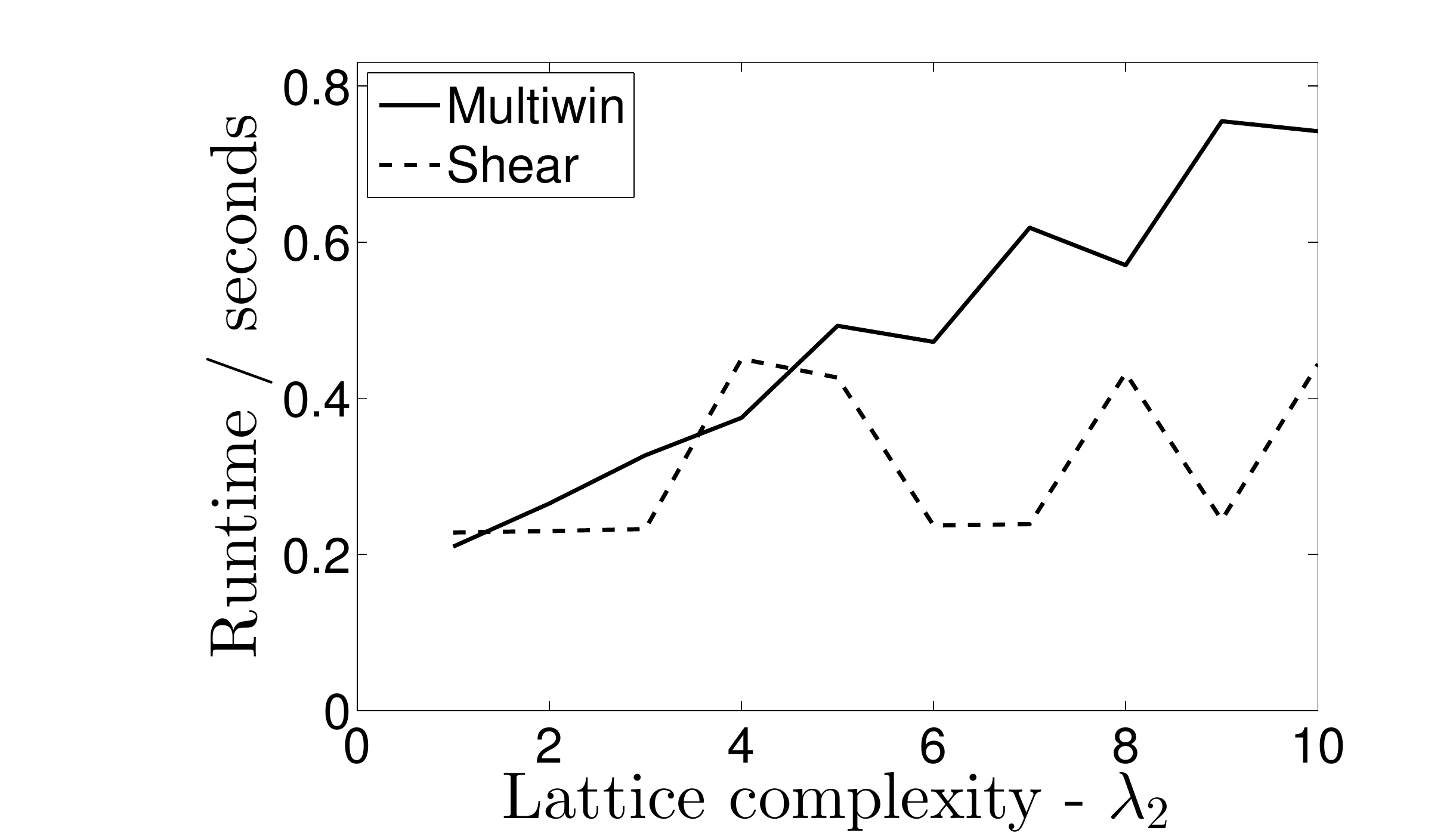}
  \caption{\label{fig:longer}Computation of the DGT for nonseparable
    lattices with increasing lattice complexities, $\lambda_2$. The
    length is kept fixed at $L=\lcm{a,M}\cdot 2520$ which is the
    minimal legal transform length for all the tested lattices. (left) 
    Accurate flop counts while the figures and (right) the actual running time. 
    The Gabor system parameters are $a=32$, $M=64$ ($p/q=1/2$) (1st row), 
    $a=40$, $M=60$ ($p/q=2/3$) (2nd row) and $a=60$, $M=80$
    ($p/q=3/4$) (3rd row). }
    \vspace{-10pt}
\end{figure}

The experiments shown in Figure \ref{fig:longer} illustrate how the
computational complexity of the running time of the algorithm depends
on the lattice complexity $\lambda_2$: The complexity of the shear
algorithm is independent of $\lambda_2$, while the complexity of the
multiwindow algorithm grows linearly. 

The bumps in the curves for the multiwindow algorithm are due to
variations in $q_{mw}$: The multiwindow algorithm transforms the
problem into $\lambda_2$ different DGTs that should be computed on a
lattice with redundancy $q/(p\lambda_2)$. The number $q_{mw}$ is the
nominator of this written as an irreducible fraction, and
depending on $p\lambda_2$ it may be smaller than $q$.

The bumps in the curves for the shear algorithm are caused by whether
or not a frequency side shear is required for that particular lattice
configuration, and to a lesser extend whether a time-side shear is
needed. As the multiwindow algorithm is faster for simple lattices,
there is a cross-over point where the shear algorithm becomes faster,
but the cross-over point depends strongly on the exact lattice
configuration. Just considering the flop counts would predict that the
cross-over happens for a smaller value of $\lambda_2$ that what is
really the case. This is due to the fact the there are more
complicated indexing operations and memory reshuffling for the shear
algorithm than for the multiwindow algorithm, and this is not properly
reflected in the flop count.

The cross-over point where one algorithm is faster than the other is
highly dependent on the interplay between the algorithm and the
computer architecture. Experience from the ATLAS \citep{whaley04}, FFTW
\citep{frjo05} and SPIRAL \citep{deMesmay2010} projects, show that in
order to have the highest performance, is it necessary to select the
algorithm for a given problem size based on previous tests done on the
very same machine. Performing such an optimization is beyond the scope
of this paper, and we therefore cannot make statements about how to
choose the most efficient cross-over points.

\section*{Acknowledgment}
This research was supported by the Austrian Science Fund (FWF) START-project FLAME (``Frames and Linear Operators for Acoustical Modeling and Parameter Estimation''; Y 551-N13) and the 
EU FET Open grant UNLocX (255931).

%\bibliographystyle{IEEEtranS}
%\bibliography{note019,nhgbib}

\begin{thebibliography}{10}
\providecommand{\url}[1]{#1}
\csname url@rmstyle\endcsname
\providecommand{\newblock}{\relax}
\providecommand{\bibinfo}[2]{#2}
\providecommand\BIBentrySTDinterwordspacing{\spaceskip=0pt\relax}
\providecommand\BIBentryALTinterwordstretchfactor{4}
\providecommand\BIBentryALTinterwordspacing{\spaceskip=\fontdimen2\font plus
\BIBentryALTinterwordstretchfactor\fontdimen3\font minus
  \fontdimen4\font\relax}
\providecommand\BIBforeignlanguage[2]{{%
\expandafter\ifx\csname l@#1\endcsname\relax
\typeout{** WARNING: IEEEtran.bst: No hyphenation pattern has been}%
\typeout{** loaded for the language `#1'. Using the pattern for}%
\typeout{** the default language instead.}%
\else
\language=\csname l@#1\endcsname
\fi
#2}}

\bibitem{ltfatweb}
``{LTFAT - The Large Time-Frequency Analysis Toolbox},''
  {http://ltfat.sourceforge.net/}.

\bibitem{allen1977unified}
J.~Allen and L.~Rabiner, ``{A unified approach to short-time Fourier analysis
  and synthesis},'' \emph{Proceedings of the IEEE}, vol.~65, no.~11, pp.
  1558--1564, 1977.

\bibitem{auslander1991discrete}
L.~Auslander, I.~Gertner, and R.~Tolimieri, ``{The discrete Zak transform
  application to time-frequency analysis and synthesis of nonstationary
  signals},'' \emph{IEEE\ Trans.\ Signal\ Process.}, vol.~39, no.~4, pp.
  825--835, 1991.

\bibitem{bage96}
M.~J. {B}astiaans and M.~C.~W. {G}eilen, ``{O}n the discrete {G}abor transform
  and the discrete {Z}ak transform,'' \emph{Signal Process.}, vol.~49, no.~3,
  pp. 151--166, 1996.

\bibitem{bastiaans1998rectangular}
M.~J. Bastiaans and A.~J. van Leest, ``From the rectangular to the quincunx
  {G}abor lattice via fractional {F}ourier transformation,'' \emph{Signal
  Processing Letters, IEEE}, vol.~5, no.~8, pp. 203--205, 1998.

\bibitem{bastiaans1998modified}
------, ``Modified {Z}ak transform for the
  quincunx-type {G}abor lattice,'' in \emph{Time-Frequency and Time-Scale
  Analysis, 1998. Proceedings of the IEEE-SP International Symposium on}.\hskip
  1em plus 0.5em minus 0.4em\relax IEEE, 1998, pp. 173--176.

\bibitem{bastiaans2001gabor}
------, ``Gabor's signal expansion and the gabor transform based on a
  non-orthogonal sampling geometry,'' in \emph{Signal Processing and its
  Applications, Sixth International, Symposium on. 2001}, vol.~1.\hskip 1em
  plus 0.5em minus 0.4em\relax IEEE, 2001, pp. 162--163.

\bibitem{bastiaans2003gabor}
------, ``Gabor's signal expansion for a non-orthogonal
  sampling geometry,'' \emph{Time-frequency signal analysis and processing: a
  comprehensive reference/Ed. B. Boashash}, p. 252--260, 2003.

\bibitem{cafi06}
P.~G. {C}asazza and M.~{F}ickus, ``\BIBforeignlanguage{english}{{F}ourier
  transforms of finite chirps.}'' \emph{\BIBforeignlanguage{english}{EURASIP J.
  Adv. Signal Process.}}, vol. 2006, pp. 1--7, 2006.

\bibitem{chfepa11}
O.~{C}hristensen, H.~G. {F}eichtinger, and S.~{P}aukner,
  \emph{\BIBforeignlanguage{english}{{G}abor {A}nalysis for {I}maging}}.\hskip
  1em plus 0.5em minus 0.4em\relax {S}pringer {B}erlin, 2011, vol.~3, pp.
  1271--1307.

\bibitem{cooley1965algorithm}
J.~Cooley and J.~Tukey, ``{An algorithm for the machine calculation of complex
  {F}ourier series},'' \emph{Math. Comput}, vol.~19, no.~90, pp. 297--301, 1965.

\bibitem{deMesmay2010}
F.~{de Mesmay}, Y.~Voronenko, and M.~P{\"u}schel, ``Offline library adaptation
  using automatically generated heuristics,'' in \emph{International Parallel
  and Distributed Processing Symposium (IPDPS)}, 2010.

\bibitem{abdo12}
M.~{D}{\"o}rfler and L.~D. {A}breu, ``\BIBforeignlanguage{english}{{A}n inverse
  problem for localization operators},''
  \emph{\BIBforeignlanguage{english}{Inverse Problems}}, vol.~28, 2012.

\bibitem{fegr92-1}
H.~G. {F}eichtinger and K.~{G}r{\"o}chenig,
  ``\BIBforeignlanguage{{E}nglish}{{G}abor wavelets and the {H}eisenberg group:
  {G}abor expansions and short time {F}ourier transform from the group
  theoretical point of view},'' in
  \emph{\BIBforeignlanguage{{E}nglish}{{W}avelets :a tutorial in theory and
  applications}}, ser. {W}avelet {A}nal. {A}ppl., C.~K. {C}hui, Ed.\hskip 1em
  plus 0.5em minus 0.4em\relax {B}oston: {A}cademic {P}ress, 1992, vol.~2, pp.
  359--397.

\bibitem{fehakamane08}
H.~G. {F}eichtinger, M.~{H}azewinkel, N.~{K}aiblinger, E.~{M}atusiak, and
  M.~{N}euhauser, ``\BIBforeignlanguage{english}{{M}etaplectic operators on
  ${C}^n$},'' \emph{\BIBforeignlanguage{english}{Quart. J. Math. Oxford Ser.}},
  vol.~59, no.~1, pp. 15--28, 2008.

\bibitem{feka97}
H.~G. {F}eichtinger and N.~{K}aiblinger, ``2{D}-{G}abor analysis based on 1{D}
  algorithms,'' in \emph{{P}roc. {O}{E}{A}{G}{M}-97 ({H}allstatt, {A}ustria)},
  1997.

\bibitem{fekapr97}
H.~G. {F}eichtinger, N.~{K}aiblinger, and P.~{P}rinz, ``{A} {P}{O}{C}{S}
  approach to {G}abor analysis,'' in \emph{{D}{I}{P}-97 ({V}ienna, {A}ustria)},
  ser. {S}{P}{I}{E}, vol. 3346, {O}ctober 1997, pp. 18--29.

\bibitem{fekoprst96}
H.~G. {F}eichtinger, W.~{K}ozek, P.~{P}rinz, and T.~{S}trohmer, ``{O}n
  multidimensional non-separable {G}abor expansions,'' in \emph{{P}roc.
  {S}{P}{I}{E}: {W}avelet {A}pplications in {S}ignal and {I}mage {P}rocessing
  {I}{V}}, {A}ugust 1996.

\bibitem{frjo05}
M.~Frigo and S.~G. Johnson, ``The design and implementation of {FFTW3},''
  \emph{Proceedings of the IEEE}, vol.~93, no.~2, pp. 216--231, 2005, special
  issue on "Program Generation, Optimization, and Platform Adaptation".

\bibitem{gr01}
K.~Gr{\"o}chenig, \emph{Foundations of Time-Frequency Analysis}.\hskip 1em plus
  0.5em minus 0.4em\relax Birkh{\"a}user, 2001.

\bibitem{hahotowi12}
M.~{H}ampejs, N.~{H}olighaus, L.~{T}{\'o}th, and C.~{W}iesmeyr,
  ``\BIBforeignlanguage{english}{{O}n the subgroups of the group ${Z}_m \times
  {Z}_n$},'' \emph{preprint, arXiv:1211.1797}, 2013.

\bibitem{helms1967fast}
H.~Helms, ``{Fast Fourier transform method of computing difference equations
  and simulating filters},'' \emph{IEEE Transactions on Audio and
  Electroacoustics}, vol.~15, no.~2, pp. 85--90, 1967.

\bibitem{ltfatnote007}
\BIBentryALTinterwordspacing
A.~J. E.~M. Janssen and P.~L. S{\o}ndergaard, ``{Iterative algorithms to
  approximate canonical Gabor windows: Computational aspects},'' \emph{J.\
  Fourier Anal.\ Appl.}, vol.~13, no.~2, pp. 211--241, 2007.
\BIBentrySTDinterwordspacing

\bibitem{jofr07}
S.~Johnson and M.~Frigo, ``{A Modified Split-Radix FFT With Fewer Arithmetic
  Operations},'' \emph{IEEE\ Trans.\ Signal\ Process.}, vol.~55, no.~1, p. 111,
  2007.

\bibitem{ka99-1}
N.~{K}aiblinger, ``{M}etaplectic representation, eigenfunctions of phase space
  shifts, and {G}elfand-{S}hilov spaces for {L}{C}{A} groups,'' Ph.D.
  dissertation, {D}ept. {M}athematics, {U}niv. {V}ienna, 1999.

\bibitem{kane09}
N.~{K}aiblinger and M.~{N}euhauser,
  ``\BIBforeignlanguage{english}{{M}etaplectic operators for finite abelian
  groups and ${R}^d$},'' \emph{\BIBforeignlanguage{english}{Indag. Math.}},
  vol.~20, no.~2, pp. 233--246, 2009.

\bibitem{pa07}
S.~{P}aukner, ``{F}oundations of {G}abor {A}nalysis for {I}mage {P}rocessing,''
  Master's thesis, 2007.

\bibitem{po76}
M.~Portnoff, ``Implementation of the digital phase vocoder using the fast
  {Fourier} transform,'' \emph{IEEE\ Trans.\ Acoust.\ Speech\ Signal\
  Process.}, vol.~24, no.~3, pp. 243--248, 1976.

\bibitem{pr96}
P.~Prinz, ``Calculating the dual {Gabor} window for general sampling sets,''
  \emph{IEEE\ Trans.\ Signal\ Process.}, vol.~44, no.~8, pp. 2078--2082, 1996.

\bibitem{schafer1973design}
R.~Schafer and L.~Rabiner, ``{Design and Simulation of a Speech
  Analysis-Synthesis System based on Short-Time Fourier Analysis},''
  \emph{{IEEE Trans. Audio Electroac.}}, vol.~21, no.~3, pp.
  165--174, 1973.

\bibitem{ltfatnote015}
P.~L. S{\o}ndergaard, B.~Torr\'esani, and P.~Balazs, ``{The Linear Time
  Frequency Analysis Toolbox},'' \emph{International Journal of Wavelets,
  Multiresolution Analysis and Information Processing}, vol.~10, no.~4, 2012.

\bibitem{ltfatnote011}
P.~L. S{\o}ndergaard, ``{Efficient Algorithms for the Discrete Gabor Transform
  with a long FIR window},'' \emph{J.\ Fourier Anal.\ Appl.}, vol.~18, no.~3,
  pp. 456--470, 2012.

\bibitem{stockham1966high}
T.~Stockham~Jr, ``{High-speed convolution and correlation},'' in
  \emph{Proc. SJCC, 1966}.\hskip 1em plus 0.5em minus 0.4em\relax ACM, 1966, pp. 229--233.

\bibitem{st98-8}
T.~Strohmer, ``\BIBforeignlanguage{{E}nglish}{Numerical algorithms for discrete
  {Gabor} expansions},'' in \emph{\BIBforeignlanguage{{E}nglish}{{Gabor}
  Analysis and Algorithms}}.\hskip 1em plus 0.5em minus 0.4em\relax 
  {B}irkh{\"a}user, 1998, ch.~8, pp. 267--294.

\bibitem{best03}
T.~{S}trohmer and S.~{B}eaver, ``\BIBforeignlanguage{english}{{O}ptimal
  {O}{F}{D}{M} system design for time-frequency dispersive channels},''
  \emph{\BIBforeignlanguage{english}{IEEE Trans. Comm.}}, vol.~51, no.~7, pp.
  1111--1122, {J}uly 2003.

\bibitem{va01-2}
A.~J. van {L}eest, ``\BIBforeignlanguage{english}{{N}on-separable {G}abor
  schemes. {T}heir {D}esign and {I}mplementation},'' Ph.D. dissertation,
  {T}ech. {U}niv. {E}indhoven, 2001.

\bibitem{van2000gabor}
A.~J. van Leest and M.~J. Bastiaans, ``Gabor's discrete signal expansion and
  the discrete {G}abor transform on a non-separable lattice,'' in
  \emph{Proc. ICASSP'00.}, vol.~1.\hskip 1em plus 0.5em minus
  0.4em\relax IEEE, 2000, pp. 101--104.

\bibitem{bava04}
------,
  ``\BIBforeignlanguage{english}{{I}mplementations of non-separable {G}abor
  schemes},'' in \emph{\BIBforeignlanguage{english}{{P}roc. EUSIPCO 2004
  ,{V}ienna,{A}ustria,}}, 2004, pp. 1565--1568.

\bibitem{we64}
A.~{W}eil, ``{S}ur certains groupes d'op{\'e}rateurs unitaires,'' \emph{Acta
  Math.}, vol. 111, pp. 143--211, 1964.

\bibitem{whaley04}
R.~C. Whaley and A.~Petitet, ``Minimizing development and maintenance costs in
  supporting persistently optimized {BLAS},'' \emph{Software: Practice and
  Experience}, vol.~35, no.~2, pp. 101--121, February 2005.

\bibitem{zezi93}
Y.~Y. Zeevi and M.~Zibulski, ``Oversampling in the {Gabor} scheme,''
  \emph{IEEE\ Trans.\ Signal\ Process.}, vol.~41, no.~8, pp. 2679--2687, 1993.

\bibitem{zezi96}
M.~{Z}ibulski and Y.~Y. {Z}eevi, ``\BIBforeignlanguage{english}{{S}ignal- and
  image-component separation by a multi-window {G}abor-type scheme},'' in
  \emph{\BIBforeignlanguage{english}{Proc. ICPR, 1996}}, vol.~2.\hskip
  1em plus 0.5em minus 0.4em\relax {V}ienna , {A}ustria, pp. 835 --839.

\bibitem{zezi97}
------, ``{A}nalysis of multiwindow {G}abor-type
  schemes by frame methods,'' \emph{Appl. Comput. Harmon. Anal.}, vol.~4,
  no.~2, pp. 188--221, 1997.

\bibitem{zezi97-1}
------, ``{D}iscrete multiwindow {G}abor-type transforms.'' \emph{IEEE Trans.
  Signal Process.}, vol.~45, no.~6, pp. 1428--1442, 1997.

\bibitem{zezi98}
------, ``\BIBforeignlanguage{english}{{T}he generalized {G}abor scheme and its
  application in signal and image representation},'' in
  \emph{\BIBforeignlanguage{english}{{S}ignal and {I}mage {R}epresentation in
  {C}ombined {S}paces}}, ser. {W}avelet {A}nal. {A}ppl.\hskip 1em plus 0.5em
  minus 0.4em\relax {A}cademic {P}ress, 1998, vol.~7,
  pp. 121--164.

\end{thebibliography}

\end{document}